\documentclass[11pt]{article}
\usepackage{amsfonts, amsmath, amssymb, latexsym, eucal, amscd} 
\usepackage{cite}
\usepackage[dvips]{pict2e}

\usepackage[dvipdfmx]{graphics}
\newtheorem{theorem}{Theorem}[section]
\newtheorem{lemma}[theorem]{Lemma}
\newtheorem{proposition}[theorem]{Proposition}

\newtheorem{exAux}[theorem]{Example}

\newtheorem{Def}[theorem]{Definition}
\newenvironment{definition}{\begin{Def} \rm}{\end{Def}}
\newtheorem{Note}[theorem]{Note}

\newtheorem{Problem}[theorem]{Problem}

\newtheorem{Rem}[theorem]{Remark}

\newtheorem{Not}[theorem]{Notation}

\newtheorem{Conj}[theorem]{Conjecture}

\newtheorem{Ass}[theorem]{Assumption}

\newenvironment{proof}{\medskip\noindent{\bf Proof.\ }}{\qed\medskip}

\newcommand{\qed}{\hfill\mbox{$\Box$\qquad\qquad}}

\newcommand{\F}{\mathbb{F}}

\newcommand{\Mat}{\text{\rm Mat}}

\renewcommand{\b}[1]{\langle #1 \rangle }

\newcommand{\bbig}[1]{\bigl\langle #1 \bigr\rangle}

\renewcommand{\th}{\theta}

\newcommand{\R}{\mathbb{R}}
\newcommand{\V}{{\mathbb V}}

\newif\ifDRAFT

\begin{document}

\title{Idempotent systems}

\author{Kazumasa Nomura and Paul Terwilliger}

\maketitle

\medskip

\begin{quote}
\small 
\begin{center}
\bf Abstract
\end{center}

In this paper we introduce the notion of an idempotent system.
This linear algebraic object is motivated by the structure of an association scheme.
We focus on a family of idempotent systems, said to be symmetric.
A symmetric idempotent system is an abstraction of the
primary module for the subconstituent algebra of a symmetric association scheme.
We describe the symmetric idempotent systems in detail.
We also consider a class of symmetric idempotent systems,
said to be $P$-polynomial and $Q$-polynomial.
In the topic of orthogonal polynomials there is an object called a Leonard system.
We show that a Leonard system is essentially the same thing as a symmetric
idempotent system that is $P$-polynomial and $Q$-polynomial.
\end{quote}

\section{Introduction}
\label{sec:intro}

In this paper we introduce the notion of an idempotent system.
This  linear algebraic object is motivated by the structure of an
association scheme.
Before summarizing the contents of this paper, we briefly recall
the notion of an association scheme.
A (symmetric) association scheme is a sequence $(X, \{R_i\}_{i=0}^d)$,
where $X$ is a finite nonempty set,
and $\{R_i\}_{i=0}^d$ is a sequence of nonempty subsets of $X \times X$ 
such that
\begin{itemize}
\item[(i)]
$X \times X = R_0 \cup R_1 \cup \cdots \cup R_d \quad \text{(disjoint union)}$;
\item[(ii)]
$R_0 = \{(x,x) \,|\, x \in X\}$;
\item[(iii)]
$(x,y) \in R_i$ implies $(y,x) \in R_i$; 
\item[(iv)]
there exist integers $p^h_{i j}$ $(0 \leq h,i,j \leq d)$ such that
for any $(x,y) \in R_h$ the number of $z \in X$ with $(x,z)\in R_i$ and $(z,y) \in R_j$
is equal to $p^h_{i j}$.
\end{itemize}
The integers $p^h_{i j}$ are called the intersection numbers.
By (iii) they satisfy $p^h_{i j} = p^h_{j i}$ for $0 \leq h,i,j \leq d$.
The concept of an association scheme first arose in design theory 
\cite{BN, BS, BM, RC} and group theory \cite{Wi}.
A systematic study began with \cite{Del, Hig}.
A comprehensive treatment is given in \cite{BI, BCN}.

Let $(X, \{R_i\}_{i=0}^d)$ denote an association scheme.
As we study this object, the following concepts and notation will be useful.
Let $\R$ denote the real number field.
Let $\Mat_X(\R)$ denote the $\R$-algebra consisting of  the matrices with
rows and columns indexed by $X$, and all entries in $\R$.
Let $I$ (resp.\ $J$) denote the identity matrix (resp.\ all $1$'s matrix) in $\Mat_X(\R)$.
Let $\V$ denote the vector space over $\R$ consisting of the column vectors with coordinates
indexed by $X$, and all entries in $\R$. 
The algebra $\Mat_X(\R)$ acts on $\V$ by left multiplication.
We define a bilinear form $\b{ \; , \; } : \V \times \V \to \R$
such that $\b{u,v} = \sum_{y \in X} u_y v_y$ for $u,v \in \V$.
We have $\b{B u, v} = \b{u, B^{\sf t} v}$ for $B \in \Mat_X(\R)$ and $u,v \in \V$.
Here $B^{\sf t}$ denotes the transpose of $B$.
For $y \in X$ define $\widehat{y} \in \V$ that has $y$-entry $1$ and all other entries $0$.
Note that $\{ \widehat{y} \; | \; y \in X\}$ form an orthonormal basis of $\V$.

We now recall the Bose-Mesner algebra.
For $0 \leq i \leq d$ define $A_i \in \Mat_X(\R)$ that has $(y,z)$-entry $1$
if $(y,z) \in R_i$ and $0$ if $(y,z) \not\in R_i$  $(y,z \in X)$.
The matrix $A_i$ is symmetric.
We have 
\begin{align*}
A_0 &= I,  &  A_i A_j &=  \sum_{h=0}^d p^h_{i j} A_h \qquad (0 \leq i,j \leq d).
\end{align*}
The $\{A_i\}_{i=0}^d$ form a basis for a commutative subalgebra $M$ of $\Mat_X(\R)$.
We call $M$ the Bose-Mesner algebra of the scheme.
Each matrix in $M$ is symmetric.
By \cite[Section 2.2]{BCN}
there exists a basis  $\{E_i\}_{i=0}^d$ for $M$ such that 
\begin{align*}
E_0 &= |X|^{-1} J, &
I &= \sum_{i=0}^d E_i,  &
E_i E_j &= \delta_{i,j} E_i  \quad (0 \leq i,j \leq d).
\end{align*}
We have
\[
  \V = \sum_{i=0}^d E_i \V   \qquad\qquad \text{\rm (orthogonal direct sum)}.
\]
For $0 \leq i \leq d$, $E_i \V$ is the $i^\text{th}$ common eigenspace for $M$,
and $E_i$ is the orthogonal projection from $\V$ onto $E_i \V$.
There exist real numbers $p_i (j)$, $q_i (j)$ $(0 \leq i,j \leq d)$ such that
\begin{align*}
  A_i &= \sum_{j=0}^d p_i (j) E_j,  & 
  E_i &= |X|^{-1} \sum_{j=0}^d q_i (j) A_j
\end{align*}
for $0 \leq i \leq d$.

We now recall the Krein parameters.
Note that $A_i \circ A_j = \delta_{i,j} A_i$  $(0 \leq i,j \leq d)$,
where $\circ$ denotes entry-wise multiplication.
Therefore $M$ is closed under $\circ$.
Consequently there exist real numbers $q^h_{i j}$ $(0 \leq h,i,j \leq d)$ such that 
\begin{align*}
E_i \circ E_j  &= |X|^{-1} \sum_{h=0}^d q^h_{i j} E_h    && (0 \leq i,j \leq d).
\end{align*}
By \cite[Theorem 3.8]{BI}, $q^h_{i j} \geq 0$ for $0 \leq h,i,j \leq d$.
The $q^h_{i j}$ are called the Krein parameters of the scheme.

We now recall the dual Bose-Mesner algebra.
For the rest of this section fix $x \in X$.
For $B \in M$ let $B^\rho$ denote the diagonal matrix in $\Mat_X(\R)$
that has $(y,y)$-entry $B_{x,y}$ for $y \in X$.
Roughly speaking, $B^\rho$ is obtained by turning  column $x$ of $B$ at a 45 degree angle.
For $0 \leq i \leq d$ define $E^*_i = A_i^\rho$.
For $y \in X$ the $(y,y)$-entry of $E^*_i$ is $1$ if $(x,y) \in R_i$ and $0$
if $(x,y) \not\in R_i$.
Note that $E^*_0$ has $(x,x)$-entry $1$ and all other entries $0$.
The matrices $\{E^*_i\}_{i=0}^d$ satisfy
\begin{align*}
 I &= \sum_{i=0}^d E^*_i,    & 
 E^*_i E^*_j &= \delta_{i,j} E^*_i  \qquad (0 \leq i,j \leq d).
\end{align*}
Therefore $\{E^*_i\}_{i=0}^d$ form a basis for a commutative subalgebra $M^*$ of $\Mat_X(\R)$.
We call $M^*$ the dual Bose-Mesner algebra with respect to $x$.
We have
\[
   \V = \sum_{i=0}^d E^*_i \V \qquad\qquad \text{(orthogonal direct sum)}.
\]
For $0 \leq i \leq d$, $E^*_i \V$ has basis $\{\widehat{y} \, | \, y \in X, \; (x,y) \in R_i\}$.
Moreover $E^*_i \V$ is the $i^\text{th}$ common eigenspace for $M^*$,
and $E^*_i$ is the orthogonal projection from $\V$ onto $E^*_i \V$.

The map $\rho : M \to M^*$, $B \mapsto B^\rho$ is $\R$-linear and bijective.
For $0 \leq i \leq d$ define $A^*_i = |X| E_i^\rho$.
The $\{A^*_i\}_{i=0}^d$ form a basis of $M^*$, and
\begin{align*}
 A^*_0 &= I,  &
 A^*_i A^*_j &= \sum_{h=0}^d q^h_{i j} A^*_h \qquad (0 \leq i,j \leq d).
\end{align*}
For $0 \leq i \leq d$,
\begin{align*}
 A^*_i &= \sum_{j=0}^d q_i (j) E^*_j,
&
 E^*_i &= |X|^{-1} \sum_{j=0}^d p_i (j) A^*_j.
\end{align*}

We now recall the subconstituent algebra $T$ and the primary $T$-module.
Let $T$ denote the subalgebra of $\Mat_X(\R)$ generated by $M$ and $M^*$.
We call $T$ the subconstituent algebra (or Terwilliger algebra) with respect to $x$.
The algebra $T$ is closed under the transpose map.
By \cite[Lemma 3.4]{T:subconst1} the algebra $T$ is semisimple.
Moreover by \cite[Lemma 3.4]{T:subconst1} the $T$-module $\V$ 
decomposes into an orthogonal direct sum of irreducible $T$-modules.
Among these modules there is a distinguished one, said to be primary.
We now describe the primary $T$-module.
Let $\bf 1$ denote the vector in $\V$ that has all entries $1$.
So ${\bf 1} = \sum_{y \in X} \widehat{y}$.
For $0 \leq i \leq d$,
\begin{align*}
A_i \widehat{x} &= E^*_i {\bf 1},  &  |X|^{-1} A^*_i {\bf 1} &= E_i \widehat{x}.
\end{align*}
Therefore $M \widehat{x} = M^* {\bf 1}$; denote this common vector space by  $V$.
By construction $V$ is a $T$-module with dimension $d+1$.
By \cite[Lemma 3.6]{T:subconst1} the $T$-module $V$ is irreducible.
The $T$-module $V$ is said to be primary.
For $0 \leq i \leq d$ define 
\begin{align*}
  {\bf 1}_i &= A_i \widehat{x} = E^*_i {\bf 1}.
\end{align*}
The vector ${\bf 1}_i$ is a basis of $E^*_i V$.
Moreover $\{ {\bf 1}_i \}_{i=0}^d$ is a basis of $V$.
This basis is orthogonal and 
$|| {\bf 1}_i ||^2 = k_i$ where $k_i = \text{rank}(E^*_i)$ $(0 \leq i \leq d)$.
The basis $\{ {\bf 1}_i\}_{i=0}^d$ diagonalizes $M^*$.
For $0 \leq i,j \leq d$,
\begin{align*}
  E^*_i {\bf 1}_j &= \delta_{i,j} {\bf 1}_j,
&
  A_i {\bf 1}_j &= \sum_{h=0}^d p^h_{i j} {\bf 1}_h.
\end{align*}
For $0 \leq i \leq d$ define 
\begin{align*}
  {\bf 1}^*_i &= |X|^{-1} A^*_i {\bf 1} = E_i \widehat{x}.
\end{align*}
The vector ${\bf 1}^*_i$ is a basis of $E_i V$.
Moreover $\{ {\bf 1}^*_i \}_{i=0}^d$ is a basis of $V$.
This basis is orthogonal and
$|| {\bf 1}^*_i ||^2 = k^*_i$ where $k^*_i = \text{rank}(E_i)$ $(0 \leq i \leq d)$.
The basis $\{ {\bf 1}^*_i \}_{i=0}^d$ diagonalizes $M$.
For $0 \leq i,j \leq d$,
\begin{align*}
 E_i {\bf 1}^*_j &= \delta_{i,j} {\bf 1}^*_j,
&
 A^*_i {\bf 1}^*_j &= \sum_{h=0}^d q^h_{i j} {\bf 1}^*_h.
\end{align*}
The bases $\{ {\bf 1}_i \}_{i=0}^d$ and $\{ {\bf 1}^*_i \}_{i=0}^d$ are related by
\begin{align*}
{\bf 1}_i &= \sum_{j=0}^d p_i (j) {\bf 1}^*_j,
&
{\bf 1}^*_i &= |X|^{-1} \sum_{j=0}^d q_i (j) {\bf 1}_j
\end{align*}
for $0 \leq i \leq d$.
The following bases for $V$ are of interest:
\begin{align*}
& \text{\rm (i)} \;\; \{ {\bf 1}_i\}_{i=0}^d,
&& \text{\rm (ii)} \;\; \{ k_i^{-1} {\bf 1}_i \}_{i=0}^d,
&& \text{\rm (iii)} \;\; \{ {\bf 1}^*_i \}_{i=0}^d,
&& \text{\rm (iv)} \;\; \{ |X| (k^*_i)^{-1} {\bf 1}^*_i \}_{i=0}^d.
\end{align*}
The bases (i), (ii) are dual with respect to $\b{ \; ,\; }$.
Moreover the bases (iii), (iv) are dual with respect to $\b{ \; , \; }$.

The algebras $M$ and $M^*$ are related as follows.
For $0 \leq h,i,j \leq d$,
\begin{align*}
  E^*_i A_j E^*_h &= 0 \qquad \text{if and only if} \qquad p^h_{i j}=0,
\\
  E_i A^*_j E_h &= 0 \qquad \text{if and only if} \qquad q^h_{i j} = 0.
\end{align*}
For $0 \leq i \leq d$,
\begin{align*}
A_i E^*_0 E_0 &= E^*_i E_0, &
A^*_i E_0 E^*_0 &= E_i E^*_0.
\end{align*}
For $0 \leq i \leq d$,
\begin{align*}
 E_0 E^*_i &\neq0, &
 E^*_0 E_i & \neq 0, &
 E^*_i E_0 &\neq 0, &
 E_i E^*_0 &\neq 0.
\end{align*}

We summarize the above description with four statements about $V$:
\begin{itemize}
\item[\rm (i)]
the $\{E_i\}_{i=0}^d$ act on $V$ as a system of mutually orthogonal rank $1$ idempotents;
\item[\rm (ii)]
the $\{E^*_i\}_{i=0}^d$ act on $V$ as a system of mutually orthogonal rank $1$ idempotents;
\item[\rm (iii)]
$E_0 E^*_i E_0$ is nonzero on $V$ for $0 \leq i \leq d$;
\item[\rm (iv)]
$E^*_0 E_i E^*_0$ is nonzero on $V$ for $0 \leq i \leq d$.
\end{itemize}
The above statements (i)--(iv) have the following significance.
We will show that (i)--(iv) together with the symmetry of the matrices 
$\{E_i\}_{i=0}^d$, $\{E^*_i\}_{i=0}^d$
are sufficient to recover the $T$-module $V$ at an algebraic level.

We now turn our attention to idempotent systems.
An idempotent system is defined as follows.
Let $\F$ denote a field.
Let $d$ denote a nonnegative integer,
and let $V$ denote a vector space over $\F$ with dimension $d+1$.
Let $\text{\rm End}(V)$ denote the $\F$-algebra consisting of the $\F$-linear maps from $V$ to $V$.
An idempotent system on $V$ is a sequence 
$\Phi = (\{E_i\}_{i=0}^d; \{E^*_i\}_{i=0}^d)$ such that
\begin{itemize}
\item[\rm (i)]
$\{E_i\}_{i=0}^d$ is a system of mutually orthogonal rank $1$ idempotents in $\text{\rm End}(V)$;
\item[\rm (ii)]
$\{E^*_i\}_{i=0}^d$ is a system of mutually orthogonal rank $1$ idempotents in $\text{\rm End}(V)$;
\item[\rm (iii)]
$E_0 E^*_i E_0 \neq 0 \;\; (0 \leq i \leq d)$;
\item[\rm (iv)]
$E^*_0 E_i E^*_0 \neq 0 \;\; (0 \leq i \leq d)$.
\end{itemize}
The above idempotent system $\Phi$ is said to be symmetric whenever there exists an antiautomorphism $\dagger$
of $\text{\rm End}(V)$ that fixes each of $E_i$, $E^*_i$ for $0 \leq i \leq d$.
The map $\dagger$ corresponds to the transpose map.

Let $\Phi = (\{E_i\}_{i=0}^d; \{E^*_i\}_{i=0}^d)$ denote a symmetric idempotent system on $V$.
Using $\Phi$ we will define some elements $\{A_i\}_{i=0}^d$, \{$A^*_i\}_{i=0}^d$ in $\text{\rm End}(V)$ 
and some scalars
\begin{equation}
\nu, \quad 
k_i, \quad 
k^*_i, \quad 
p^h_{i j}, \quad
q^h_{i j}, \quad
p_i (j), \quad
q_i (j)                         \label{eq:scalars}
\end{equation}
in $\F$.
The scalar $\nu$ corresponds to $|X|$.
We will endow $V$ with a nondegenerate symmetric bilinear form $\b{ \; , \;}$.
We will define four orthogonal bases of $V$ that correspond to the four earlier bases of interest.
We will show that the resulting construction matches the primary $T$-module at an algebraic level.

Our definitions are summarized as follows.
Note that $\{E_i\}_{i=0}^d$ form a basis for a commutative subalgebra $\cal M$ of $\text{\rm End}(V)$.
We show that for $0 \leq i \leq d$ there exists a unique $A_i \in {\cal M}$ such that 
\[
   A_i E^*_0 E_0 = E^*_i E_0.
\]
We show that $\{A_i\}_{i=0}^d$ is a basis for the vector space $\cal M$.
Similarly, the $\{E^*_i\}_{i=0}^d$ form a basis for a commutative subalgebra ${\cal M}^*$ of $\text{\rm End}(V)$.
We show that for $0 \leq i \leq d$ there exists a unique $A^*_i \in {\cal M}^*$ such that 
\[
   A^*_i E_0 E^*_0 = E_i E^*_0.
\]
We show that $\{A^*_i\}_{i=0}^d$ is a basis for the vector space ${\cal M}^*$.

Concerning the scalars \eqref{eq:scalars}, 
we show that $\text{\rm tr}(E_0 E^*_0) \neq 0$.
The scalar $\nu$ is defined by
\begin{align*}
\nu &= \text{\rm tr}(E_0 E^*_0)^{-1}.
\end{align*}
The scalars $k_i$, $k^*_i$ are defined by
\begin{align*}
k_i &= \nu \, \text{\rm tr}(E_0 E^*_i),   &
k^*_i &= \nu \, \text{\rm tr}(E^*_0 E_i)  &&  (0 \leq i \leq d).
\end{align*}
We show that  $\sum_{i=0}^d k_i = \nu = \sum_{i=0}^d k^*_i$,
and each of $k_i$, $k^*_i$ is nonzero for $0 \leq i \leq d$.
The scalars $p^h_{i j}$, $q^h_{i j}$ are defined by
\begin{align*}
  A_i A_j &= \sum_{h=0}^d p^h_{i j} A_h,
&
 A^*_i A^*_j &= \sum_{h=0}^d q^h_{ i j} A^*_h   && (0 \leq i,j \leq d).
\end{align*}
The scalars $p_i (j)$, $q_i (j)$ are defined by
\begin{align*}
   A_i &= \sum_{j=0}^d p_i (j) E_j,
&
 A^*_i &= \sum_{j=0}^d q_i (j) E^*_j   &&   (0 \leq i \leq d).
\end{align*}

We define a bilinear form $\b{ \; , \; }$ on $V$ as follows.
By linear algebra, there exists a nondegenerate bilinear form $\b{ \; , \; }$ on $V$
such that $\b{B u, v} = \b{ u, B^\dagger v}$ for all $B \in \text{\rm End}(V)$ and $u, v \in V$.
The bilinear form $\b{ \; , \;}$ is unique up to multiplication by a nonzero scalar in $\F$.
The bilinear form $\b{ \; , \;}$ is symmetric.

Fix nonzero $\xi, \zeta$ in $E_0 V$ and nonzero $\xi^*, \zeta^*$ in $E^*_0 V$.
We show that each of the following (i)--(iv) is an orthogonal basis for $V$:
\begin{align*}
& \text{\rm (i)} \;\; \{ E^*_i \xi\}_{i=0}^d,
&& \text{\rm (ii)} \;\; \{ k_i^{-1} E^*_i \zeta \}_{i=0}^d,
&& \text{\rm (iii)} \;\; \{ E_i \xi^* \}_{i=0}^d,
&& \text{\rm (iv)} \;\; \{  (k^*_i)^{-1} E_i \zeta^* \}_{i=0}^d.
\end{align*}
The bases (i), (ii) are dual if and only if $\b{\xi, \zeta} = \nu$,
and the bases (iii), (iv) are dual if and only if $\b{\xi^*, \zeta^*} = \nu$.

We just summarized our definitions.
In the main body of the paper, we show that the resulting defined objects are related 
in a manner that matches the primary $T$-module.
To describe this relationship,
we use some equations involving the $\{E_i\}_{i=0}^d$, $\{E^*_i\}_{i=0}^d$,
$\{A_i\}_{i=0}^d$, $\{A^*_i\}_{i=0}^d$ called the reduction rules.

Near the end of the paper we introduce the $P$-polynomial and $Q$-polynomial properties
for symmetric idempotent systems.
We show that a symmetric idempotent system that is $P$-polynomial and $Q$-polynomial
is essentially the same thing as a Leonard system in the sense of  \cite[Definition 4.1]{T:Leonard}.

The paper is organized as follows.
In Section \ref{sec:pre} we recall some basic results from linear algebra.
In Section \ref{sec:ips} we introduce the concept of an idempotent system.
In Section \ref{sec:pi} we introduce the scalar $\nu$ and discuss some related topics.
In Section \ref{sec:symips} we introduce the symmetric idempotent systems.
In Sections \ref{sec:rho}, \ref{sec:Ai} we introduce a certain linear bijection $\rho : {\cal M} \to {\cal M}^*$
and use it to define the elements $A_i$, $A^*_i$.
In Sections \ref{sec:ki}, \ref{sec:reduction} we introduce the scalars $k_i$, $k^*_i$ and
obtain some reduction rules involving these scalars. 
In Sections \ref{sec:phij}, \ref{sec:red2} we introduce the scalars $p^h_{i j}$, $q^h_{i j}$
and obtain some reduction rules involving these scalars.
In Sections \ref{sec:pij}, \ref{sec:red3} we introduce the scalars $p_i (j)$, $q_i (j)$ 
and obtain some reduction rules involving these scalars.
In Section \ref{sec:matrices} we put some of our earlier results in matrix form.
In Sections \ref{sec:standard}--\ref{sec:dualstandard} we introduce the four bases of interest
and discuss their properties.
In Section \ref{sec:4bases} we obtain  the transition matrices between these four bases, and the
inner products between these four bases.
We also obtain  the matrices
representing $A_i$, $A^*_i$, $E_i$, $E^*_i$ with respect to these four bases.
In Section \ref{sec:Ppoly} we introduce the $P$-polynomial and $Q$-polynomial properties.
In Section \ref{sec:LP} we recall the notion of a Leonard pair and a Leonard system.
In Section \ref{sec:IPSLS} we show that a Leonard system is essentially the same thing
as a symmetric idempotent system that is $P$-polynomial and $Q$-polynomial.

The reader might wonder how the concept of a symmetric idempotent system is
related to the concept of a character algebra \cite{Kawada}.
Roughly speaking, a symmetric idempotent system is obtained by gluing together
a character algebra and its dual;
we will discuss this in a future paper.

\section{Preliminaries}
\label{sec:pre}

In this section we fix some notation and recall some basic concepts.
Throughout this paper $\F$ denotes a field.
By a {\em scalar} we mean an element of $\F$.
All algebras and vector spaces discussed in this paper are over $\F$.
All algebras discussed in this paper are associative and have a multiplicative identity.
For an algebra $\cal A$, 
by an {\em automorphism} of $\cal A$ we mean an algebra isomorphism ${\cal A} \to {\cal A}$,
and by an {\em antiautomorphism} of $\cal A$ we mean an $\F$-linear bijection
$\tau : {\cal A} \to {\cal A}$ such that $(Y Z)^\tau = Z^\tau Y^\tau$ for $Y, Z \in {\cal A}$.
For the rest of this paper, fix an integer $d \geq 0$ and let $V$ denote a vector space with dimension $d+1$.
Let $\text{\rm End}(V)$ denote the algebra consisting of the $\F$-linear
maps from $V$ to $V$. 
Let $\Mat_{d+1}(\F)$ denote the algebra consisting of the $d+1$ by $d+1$
matrices that have all entries in $\F$.
We index the rows and columns by $0,1,\ldots, d$.
The identity of $\text{\rm End}(V)$ or $\Mat_{d+1}(\F)$ is denoted by $I$.
For $A \in \text{\rm End}(V)$,
the dimension of $A V$ is called the {\em rank of $A$}.
A matrix $M \in \Mat_{d+1}(\F)$ is said to be {\em tridiagonal} whenever
the $(i,j)$-entry $M_{i,j}=0$ if $|i-j|>1$ $(0 \leq i,j \leq d)$.
Assume for the moment that $M$ is tridiagonal.
Then $M$ is said to be {\em irreducible}
whenever $M_{i,j} \neq 0$ if $|i-j|=1$ $(0 \leq i,j \leq d)$.
We recall how each basis $\{v_i\}_{i=0}^d$ of $V$ gives an algebra
isomorphism $\text{\rm End}(V) \to \Mat_{d+1}(\F)$.
For $A \in \text{\rm End}(V)$ and $M \in \Mat_{d+1}(\F)$,
we say that {\em $M$ represents $A$ with respect to $\{v_i\}_{i=0}^d$}
whenever $A v_j = \sum_{i=0}^d M_{i,j} v_i$ for $0 \leq j \leq d$.
The isomorphism sends $A$ to the unique matrix in $\Mat_{d+1}(\F)$ that represents
$A$ with respect to $\{v_i\}_{i=0}^d$.
Next we recall the transition matrix between two bases of $V$.
Let $\{u_i\}_{i=0}^d$ and $\{v_i\}_{i=0}^d$ denote bases of $V$.
By the {\em transition matrix from $\{u_i\}_{i=0}^d$ to $\{v_i\}_{i=0}^d$}
we mean the matrix $T \in \Mat_{d+1}(\F)$ such that 
$v_j = \sum_{i=0}^d T_{i,j} u_i$ for $0 \leq j \leq d$.
Let $T$ denote the transition matrix from $\{u_i\}_{i=0}^d$ to $\{v_i\}_{i=0}^d$.
Then $T$ is invertible and $T^{-1}$ is the transition matrix
from$\{v_i\}_{i=0}^d$ to $\{u_i\}_{i=0}^d$.
Let $T'$ denote the transition matrix from $\{v_i\}_{i=0}^d$ to a basis $\{w_i\}_{i=0}^d$ of $V$.
Then the transition matrix from $\{u_i\}_{i=0}^d$ to $\{w_i\}_{i=0}^d$ is $T T'$.
For $A \in \text{\rm End}(V)$ let $M$ denote the matrix representing $A$ with respect to
$\{u_i\}_{i=0}^d$.
Then $T^{-1} M T$ represents $A$ with respect to $\{v_i\}_{i=0}^d$.
Let $A \in \text{\rm End}(V)$.
A subspace $W \subseteq V$ is called an {\em eigenspace} of $A$ whenever
$W \neq 0$ and there exists a scalar $\th$ such that $W = \{ v \in V \, |\, A v = \th v\}$;
in this case $\th$ is the {\em eigenvalue} of $A$ associated with $W$.
We say that $A$ is {\em diagonalizable} whenever $V$ is spanned by the eigenspaces of $A$.
We say that $A$ is {\em multiplicity-free} whenever $A$ is diagonalizable and its eigenspaces
all have dimension one.

\begin{definition}    \label{def:decomp}    \samepage
\ifDRAFT {\rm def:decomp}. \fi
By a {\em decomposition of $V$} we mean a sequence $\{V_i\}_{i=0}^d$
consisting of one-dimensional subspaces of $V$ such that
$V = \sum_{i=0}^d V_i$ (direct sum).
\end{definition}

\begin{definition}   \label{def:orth}    \samepage
\ifDRAFT {\rm def:orth}. \fi
By a {\em system of mutually orthogonal rank $1$ idempotents} in $\text{\rm End}(V)$
we mean a sequence $\{E_i\}_{i=0}^d$ of elements in $\text{\rm End}(V)$ such that
\[
   E_i E_j = \delta_{i,j} E_i  \qquad\qquad   (0 \leq i,j \leq d),
\]
\[
  \text{\rm rank} (E_i) = 1 \qquad\qquad (0 \leq i \leq d).
\]
\end{definition}

The next lemma is routinely verified.

\begin{lemma}  \label{lem:decomp}    \samepage
\ifDRAFT {\rm lem:decomp}. \fi
The following hold.
\begin{itemize}
\item[\rm (i)]
Let $\{V_i\}_{i=0}^d$ denote a decomposition of $V$.
For $0 \leq i \leq d$ define $E_i \in \text{\rm End}(V)$ such that
$(E_i - I)V_i = 0$ and $E_i V_j = 0$ if $j \neq i$ $(0 \leq j \leq d)$.
Then $\{E_i\}_{i=0}^d$ is a system of mutually orthogonal rank $1$ idempotents in $\text{\rm End}(V)$.
\item[\rm (ii)]
Let $\{E_i\}_{i=0}^d$ denote a system of mutually orthogonal rank $1$ idempotents in $\text{\rm End}(V)$.
Then $\{E_i V\}_{i=0}^d$ is a decomposition of $V$.
\end{itemize}
\end{lemma}

\begin{definition}    \label{def:primitiveofA}    \samepage
\ifDRAFT {\rm def:primitiveofA}. \fi
Let $A$ denote a multiplicity-free element in $\text{\rm End}(V)$, and let
$\{V_i\}_{i=0}^d$ denote an ordering of the eigenspaces of $A$.
Then $\{V_i\}_{i=0}^d$ is a decomposition of $V$.
Let $\{E_i\}_{i=0}^d$ denote the corresponding system of mutually orthogonal rank $1$ idempotents
from Lemma \ref{lem:decomp}(i).
We call $\{E_i\}_{i=0}^d$ the {\em primitive idempotents of $A$}.
\end{definition}

For the rest of this section,
let $\{E_i\}_{i=0}^d$ denote a system of mutually orthogonal rank $1$
idempotents in $\text{\rm End}(V)$.
The next two  lemmas are routinely verified.

\begin{lemma}    \label{lem:trE}    \samepage
\ifDRAFT {\rm lem:trE}. \fi
The following hold:
\begin{itemize}
\item[\rm (i)]
$\text{\rm tr} (E_i) = 1$  $(0 \leq i \leq d)$, where $\text{\rm tr}$ means trace.
\item[\rm (ii)]
 $I = \sum_{i=0}^d E_i$;
\item[\rm (iii)]
$\{E_i\}_{i=0}^d$ form a basis for a commutative subalgebra of $\text{\rm End}(V)$.
\end{itemize}
\end{lemma}

\begin{lemma}  \label{lem:sumEiAEj}    \samepage
\ifDRAFT {\rm lem:sumEiAEj}. \fi
For ${\cal A} = \text{\rm End}(V)$,
\begin{itemize}
\item[\rm (i)]
the sum ${\cal A} = \sum_{i=0}^d \sum_{j=0}^d E_i {\cal A} E_j$ is direct;
\item[\rm (ii)]
$\dim E_i {\cal A} E_j = 1$ for $0 \leq i,j \leq d$.
\end{itemize}
\end{lemma}

\section{Idempotent systems}
\label{sec:ips}

Recall the vector space $V$ with dimension $d+1$.
In this section we introduce the notion of an idempotent system on $V$.

\begin{definition}    \label{def:ips}    \samepage
\ifDRAFT {\rm def:ips}. \fi
By an {\em idempotent system} on $V$ we mean a sequence
\begin{equation*}
     ( \{E_i\}_{i=0}^d; \{E^*_i\}_{i=0}^d) 
\end{equation*}
such that
\begin{itemize}
\item[\rm (i)]
$\{E_i\}_{i=0}^d$ is a system of mutually orthogonal rank $1$ idempotents in $\text{\rm End}(V)$;
\item[\rm (ii)]
$\{E^*_i\}_{i=0}^d$ is a system of mutually orthogonal rank $1$ idempotents in $\text{\rm End}(V)$;
\item[\rm (iii)]
$E_0 E^*_i E_0 \neq 0 \quad (0 \leq i \leq d)$;
\item[\rm (iv)]
$E^*_0 E_i E^*_0 \neq 0 \quad (0 \leq i \leq d)$.
\end{itemize}
\end{definition}

Let $\Phi = (\{E_i\}_{i=0}^d; \{E^*_i\}_{i=0}^d)$ denote an idempotent system on $V$.
Define
\[
   \Phi^* = ( \{E^*_i\}_{i=0}^d; \{E_i\}_{i=0}^d).
\]
Then $\Phi^*$ is an idempotent system on $V$, called the {\em dual} of $\Phi$.
We have $(\Phi^*)^* = \Phi$.
For an object $f$ attached to $\Phi$, the corresponding object attached to $\Phi^*$
is denoted by $f^*$.

Let $\Phi' = (\{E'_i\}_{i=0}^d; \{E^{* \prime}_i\}_{i=0}^d)$ denote an idempotent system on
a vector space $V'$.
By an {\em isomorphism of idempotent systems from $\Phi$ to $\Phi'$} we mean
an algebra isomorphism $\text{\rm End}(V) \to \text{\rm End}(V')$ that sends
$E_i \mapsto E'_i$, $E^*_i \mapsto E^{* \prime}_i$ for $0 \leq i \leq d$.
We say that $\Phi$ and $\Phi'$ are {\em isomorphic} whenever there exists an isomorphism
of idempotent systems from $\Phi$ to $\Phi'$.
By the Skolem-Noether theorem (see \cite[Corollary 7.125]{Rot}),
a map $\sigma : \text{\rm End}(V) \to \text{\rm End}(V')$ is an algebra isomorphism
if and only if there exists an $\F$-linear bijection $S : V \to V'$ such that $A^\sigma = S A S^{-1}$
for all $A \in \text{\rm End}(V)$.

\begin{definition}     \label{def:D}
\ifDRAFT {\rm def:D}. \fi
Let $\cal M$ denote the subalgebra of $\text{\rm End}(V)$ generated by $\{E_i\}_{i=0}^d$.
Note that $\cal M$ is commutative,
and $\{E_i\}_{i=0}^d$ form a basis of the vector space $\cal M$.
\end{definition}

\section{The scalars $m_i$, $\nu$}
\label{sec:pi}

Let $\Phi = (\{E_i\}_{i=0}^d; \{E^*_i\}_{i=0}^d)$ denote an idempotent system
on $V$.
In this section we use $\Phi$ to introduce some scalars  $\{m_i\}_{i=0}^d$, $\nu$.

\begin{definition}    \label{def:mi2}    \samepage
\ifDRAFT {\rm def:mi2}. \fi
For $0 \leq i \leq d$ define 
\begin{align}
   m_i &= \text{\rm tr} (E^*_0 E_i).       \label{eq:defmi2}
\end{align}
\end{definition}

\begin{lemma}    \label{lem:E0EsiE0}    \samepage
\ifDRAFT {\rm lem:E0EsiE0}. \fi
For $0 \leq i \leq d$ the following hold:
\begin{itemize}
\item[\rm (i)]
$E^*_0 E_i E^*_0 = m_i E^*_0$;
\item[\rm (ii)]
$E_0 E^*_i E_0 = m^*_i E_0$.
\end{itemize}
\end{lemma}

\begin{proof}
(i)
Abbreviate ${\cal A} = \text{\rm End}(V)$.
By Lemma \ref{lem:sumEiAEj}(ii),
$E^*_0$ is a basis for the vector space $E^*_0 {\cal A} E^*_0$.
So there there exists a scalar $\alpha_i$ such that $E^*_0 E_i E^*_0 =\alpha_i E^*_0$.
In this equation, take the trace of each side and simplify the result
using Lemma \ref{lem:trE}(i) and  $\text{\rm tr} (M N) = \text{\rm tr} (N M)$
to obtain $\alpha_i = m_i$.
The result follows.

(ii)
Apply (i) to $\Phi^*$.
\end{proof}

\begin{lemma}    \label{lem:EsiE0Esi}     \samepage
\ifDRAFT {\rm lem:EsiE0Esi}. \fi
For $0 \leq i \leq d$ the following hold:
\begin{itemize}
\item[\rm (i)]
$E_i E^*_0 E_i = m_i E_i$;
\item[\rm (ii)]
$E^*_i E_0 E^*_i = m^*_i E^*_i$.
\end{itemize}
\end{lemma}

\begin{proof}
Similar to the proof of Lemma \ref{lem:E0EsiE0}.
\end{proof}

\begin{lemma}    \label{lem:mi3}    \samepage
\ifDRAFT {\rm lem:mi3}. \fi
The following hold:
\begin{itemize}
\item[\rm (i)]
$m_i \neq 0$ $\quad (0 \leq i \leq d)$;
\item[\rm (ii)]
$\sum_{i=0}^d m_i = 1$.
\end{itemize}
\end{lemma}

\begin{proof}
(i)
Use Definition \ref{def:ips}(iv) and Lemma \ref{lem:E0EsiE0}(i).

(ii)
By Lemma \ref{lem:trE}(ii), $\sum_{i=0}^d E_i = I$.
In this equation, multiply each side on the left by $E^*_0$ to get
$\sum_{i=0}^d E^*_0 E_i = E^*_0$.
In this equation, take the trace of each side, and evaluate the result
using Lemma \ref{lem:trE}(i) and Definition \ref{def:mi2}.
\end{proof}

\begin{definition}    \label{def:nu}    \samepage
\ifDRAFT {\rm def:nu}. \fi
Setting $i=0$ in \eqref{eq:defmi2} we find that $m_0 = m^*_0$;
let $\nu$ denote the multiplicative inverse of this common value.
We emphasize $\nu = \nu^*$ and
\begin{equation}
    \text{\rm tr} (E_0 E^*_0) = \nu^{-1}.              \label{eq:nu}
\end{equation}
\end{definition}

\begin{lemma}  \label{lem:nuE0Es0E0}   \samepage
\ifDRAFT {\rm lem:nuE0Es0E0}. \fi
We have
\begin{align}
  \nu E_0 E^*_0 E_0 &=  E_0,  &
  \nu E^*_0 E_0 E^*_0 &= E^*_0.                        \label{eq:nu2}
\end{align}
\end{lemma}

\begin{proof}
To get the equation on the left in \eqref{eq:nu2},
set $i=0$ in Lemma \ref{lem:E0EsiE0}(ii)  and use Definition \ref{def:nu}.
Applying this to $\Phi^*$ we get the equation on the right in \eqref{eq:nu2}.
\end{proof}

\begin{lemma}    \label{lem:EsiE0Esj}    \samepage
\ifDRAFT {\rm lem:EsiE0Esj}. \fi
Each of the following is a basis of the vector space $\text{\rm End}(V)$:
\begin{itemize}
\item[\rm (i)]
$\{E_i E^*_0 E_j \,|\, 0 \leq i,j \leq d\}$;
\item[\rm (ii)]
$\{E^*_i E_0 E^*_j \,|\, 0 \leq i,j \leq d\}$.
\end{itemize}
\end{lemma}

\begin{proof}
(i)
In view of Lemma \ref{lem:sumEiAEj},
it suffices to show that $E_i E^*_0 E_j \neq 0$ for $0 \leq i,j \leq d$.
Let $i$, $j$ be given, and suppose $E_i E^*_0 E_j = 0$.
Using Lemmas \ref{lem:E0EsiE0}(i) and \ref{lem:mi3}(i),
\[
  0 = E^*_0 (E_i E^*_0 E_j) E^*_0 =  m_i m_j E^*_0 \neq 0
\]
for a contradiction.
The result follows.

(ii)
Apply (i) to $\Phi^*$.
\end{proof}

\begin{lemma}    \label{lem:generate}    \samepage
\ifDRAFT {\rm lem:generate}. \fi
Each of the following is a generating set for the algebra $\text{\rm End}(V)$:
\begin{itemize}
\item[\rm (i)]
$E^*_0$ and $\cal M$;
\item[\rm (ii)]
$E_0$ and ${\cal M}^*$;
\item[\rm (iii)]
$\cal M$ and ${\cal M}^*$.
\end{itemize}
\end{lemma}

\begin{proof}
(i)
By Definition \ref{def:D} and Lemma \ref{lem:EsiE0Esj}(i).

(ii)
Apply (i) to $\Phi^*$.

(iii)
By (i) above and Definition \ref{def:D}.
\end{proof}

\section{Symmetric idempotent systems}
\label{sec:symips}

We continue to discuss an idempotent system
$\Phi = (\{E_i\}_{i=0}^d; \{E^*_i\}_{i=0}^d)$ on $V$.

\begin{definition}   \label{def:sym}    \samepage
\ifDRAFT {\rm def:sym}. \fi
We say that $\Phi$ is {\em symmetric} whenever there exists an antiautomorphism
$\dagger$ of $\text{\rm End}(V)$ that fixes each of $E_i$, $E^*_i$ for $0 \leq i \leq d$. 
\end{definition}

Recall the algebra $\cal M$ from Definition \ref{def:D}.

\begin{lemma}    \label{lem:dagger0}    \samepage
\ifDRAFT {\rm lem:dagger0}. \fi
Assume that $\Phi$ is symmetric, and let $\dagger$ denote an antiautomorphism of
$\text{\rm End}(V)$ from Definition \ref{def:sym}.
Then the following hold:
\begin{itemize}
\item[\rm (i)]
$\dagger$ is unique;
\item[\rm (ii)]
$(A^\dagger)^\dagger = A$ for $A \in  \text{\rm End}(V)$;
\item[\rm (iii)]
$\dagger$ fixes every element in $\cal M$ and every element in ${\cal M}^*$.
\end{itemize}
\end{lemma}

\begin{proof}
(iii)
By Definitions \ref{def:D} and \ref{def:sym}.

(ii)
The composition $\dagger \circ \dagger$ is an automorphism of $\text{\rm End}(V)$
that fixes everything in $\cal M$ and everything in ${\cal M}^*$.
This automorphism is the identity in view of Lemma \ref{lem:generate}(iii).

(i)
Let $\mu$ denote an antiautomorphism of  $\text{\rm End}(V)$ that fixes each of $E_i$, $E^*_i$
for $0 \leq i \leq d$.
We show $\mu = \dagger$.
The composition $\dagger \circ \mu$ is an automorphism of $\text{\rm End}(V)$
that fixes everything in $\cal M$ and everything in ${\cal M}^*$.
So this automorphism is the identity.
We have $\dagger = \dagger^{-1}$ by (ii) above, so $\mu = \dagger$.
\end{proof}

\section{The map $\rho$}
\label{sec:rho}

Let $\Phi = (\{E_i\}_{i=0}^d; \{E^*_i\}_{i=0}^d)$ denote a symmetric 
idempotent system on $V$.
Recall the  algebra $\cal M$ from Definition \ref{def:D}.
In this section we introduce a certain map $\rho :{\cal M} \to {\cal M}^*$
that will play an essential role in our theory.
As we will see, $\rho$ is an isomorphism of vector spaces but not algebras.

\begin{lemma}     \label{lem:EiEs0}    \samepage
\ifDRAFT {\rm lem:EiEs0}. \fi
For  ${\cal A} = \text{\rm End}(V)$, 
\begin{itemize}
\item[\rm (i)]
the elements
$\{E_i E^*_0\}_{i=0}^d$ form a basis of ${\cal A} E^*_0$;
\item[\rm (ii)]
the elements
$\{E^*_i E_0\}_{i=0}^d$ form a basis of ${\cal A} E_0$.
\end{itemize}
\end{lemma}

\begin{proof}
(i)
By Lemmas \ref{lem:trE}(ii) and \ref{lem:sumEiAEj}(i)
the sum ${\cal A} E^*_0 = \sum_{i=0}^d E^*_i {\cal A} E^*_0$ is direct.
Each summand has dimension one by Lemma \ref{lem:sumEiAEj}(ii), so
${\cal A} E^*_0$ has dimension $d+1$.
The elements $\{E_i E^*_0\}_{i=0}^d$ are contained in ${\cal A} E^*_0$.
We show that these elements are linear independent.
For scalars $\{\alpha_i\}_{i=0}^d$ suppose $0 = \sum_{i=0}^d \alpha_i E_i E^*_0$.
For $0 \leq r \leq d$,
multiply each side of this equation on the left by $E_r$ to obtain
$0 = \alpha_r E_r E^*_0$.
We have $E_r E^*_0 \neq 0$ by Definition \ref{def:ips}(iv), so $\alpha_r = 0$.
We have shown that $\{E_i E^*_0\}_{i=0}^d$ are linearly independent, and hence
a basis of ${\cal A} E^*_0$.

(ii)
Apply (i) to $\Phi^*$.
\end{proof}

\begin{lemma}  \label{lem:DtoDEs0}    \samepage
\ifDRAFT {\rm lem:DtoDEs0}. \fi
For  ${\cal A} = \text{\rm End}(V)$, 
\begin{itemize}
\item[\rm (i)]
the map ${\cal M} \to {\cal A} E^*_0$, $Y \mapsto Y E^*_0$ is an $\F$-linear bijection;
\item[\rm (ii)]
the map ${\cal M}^* \to {\cal A} E_0$, $Y \mapsto Y E_0$ is an $\F$-linear bijection.
\end{itemize}
\end{lemma}

\begin{proof}
(i)
Clearly the map is $\F$-linear.
By Lemma \ref{lem:EiEs0}(i),
the map sends the basis $\{E_i\}_{i=0}^d$ of $\cal M$ to the basis $\{E_i E^*_0\}_{i=0}^d$
of ${\cal A} E^*_0$.
So the map is bijective.

(ii)
Apply (i) to $\Phi^*$.
\end{proof}

\begin{lemma}    \label{lem:rho}   \samepage
\ifDRAFT {\rm lem:rho}. \fi
There exists a unique $\F$-linear map $\rho : {\cal M} \to {\cal M}^*$
such that for $Y \in {\cal M}$,
\begin{align}
  Y E^*_0 E_0 &= Y^\rho E_0.       \label{eq:defrho}
\end{align}
\end{lemma}

\begin{proof}
Abbreviate ${\cal A} = \text{\rm End}(V)$.
Concerning existence, consider the $\F$-linear map $g : {\cal M} \to {\cal A} E_0$,
$Y \mapsto Y E^*_0 E_0$.
Let $\mu$ denote the map in Lemma \ref{lem:DtoDEs0}(ii).
The composition
\[
  \rho : {\cal M} \xrightarrow{\;\;\; g \;\;\; } {\cal A} E_0 \xrightarrow{\;\; \mu^{-1}\;\; } {\cal M}^*
\]
satisfies \eqref{eq:defrho}.
We have shown that $\rho$ exists.
The map $\rho$ is unique by Lemma \ref{lem:DtoDEs0}(ii).
\end{proof}

\begin{lemma}    \label{lem:rhorhospre}    \samepage
\ifDRAFT {\rm lem:rhorhospre}. \fi
The maps $\rho$ and $\nu \rho^*$ are inverses.
In particular, the maps $\rho$, $\rho^*$ are bijective.
\end{lemma}
 
\begin{proof}
Pick $Y \in {\cal M}$.
Using Lemma \ref{lem:nuE0Es0E0} and applying \eqref{eq:defrho}
to both $\Phi$ and $\Phi^*$,
\[
  (Y^\rho)^{\rho^*} E^*_0
 = Y^\rho E_0 E^*_0
 = Y E^*_0 E_0 E^*_0
 = \nu^{-1} Y E^*_0.
\]
By this and Lemma \ref{lem:DtoDEs0}(i) we get $(Y^\rho)^{\rho^*} = \nu^{-1} Y$.
Applying this to $\Phi^*$,
$(Z^{\rho^*})^\rho = \nu^{-1} Z$ for $Z \in {\cal M}^*$.
Thus the maps $\rho$ and $\nu \rho^*$ are inverses.
\end{proof}

\begin{lemma}    \label{lem:rhoI}    \samepage
\ifDRAFT {\rm lem:rhoI}. \fi
The map
$\rho$ sends $I \mapsto E^*_0$ and $E_0 \mapsto \nu^{-1} I$.
\end{lemma}

\begin{proof}
Using Lemma \ref{lem:rho},
$E^*_0 E_0 = I E^*_0 E_0 = I^\rho E_0$.
This forces $E^*_0 = I^\rho$ by Lemma \ref{lem:DtoDEs0}(ii).
Using Lemmas \ref{lem:nuE0Es0E0} and \ref{lem:rho},
$E_0^\rho E_0 = E_0 E^*_0 E_0 = \nu^{-1} E_0$.
This forces $E_0^\rho = \nu^{-1} I$ by Lemma \ref{lem:DtoDEs0}(ii).
\end{proof}

\section{The elements $A_i$}
\label{sec:Ai}

We continue to discuss a symmetric idempotent system
$\Phi = (\{E_i\}_{i=0}^d; \{E^*_i\}_{i=0}^d)$ on $V$.

\begin{definition}    \label{def:Ai}    \samepage
\ifDRAFT {\rm def:Ai}. \fi
For $0 \leq i \leq d$ define 
\begin{equation}
  A_i = \nu (E^*_i)^{\rho^*}.                                 \label{eq:defAiAsi}
\end{equation}
\end{definition}

\begin{lemma}    \label{lem:rhorhos}    \samepage
\ifDRAFT {\rm lem:rhorhos}. \fi
For $0 \leq i \leq d$,
$\rho$ sends $A_i \mapsto  E^*_i$ and $E_i \mapsto \nu^{-1} A^*_i$.
\end{lemma}

\begin{proof}
By Lemma \ref{lem:rhorhospre} and Definition \ref{def:Ai},
$A_i^\rho = \nu ((E^*_i)^{\rho^*})^\rho = E^*_i$.
Applying \eqref{eq:defAiAsi} to $\Phi^*$, $E_i^\rho = \nu^{-1} A^*_i$.
\end{proof}

\begin{lemma}       \label{lem:dagger4}   \samepage
\ifDRAFT {\rm lem:dagger4}. \fi
The antiautomorphism $\dagger$ from Definition \ref{def:sym} fixes each of $A_i$, $A^*_i$ for $0 \leq i \leq d$.
\end{lemma}

\begin{proof}
By Lemma \ref{lem:dagger0}(iii) and since $A_i \in {\cal M}$, $A^*_i \in {\cal M}^*$
for $0 \leq i \leq d$.
\end{proof}

\begin{lemma}    \label{lem:AiEs0E0}    \samepage
\ifDRAFT {\rm lem:AiEs0E0}. \fi
For $0 \leq i \leq d$ the following hold:
\begin{itemize}
\item[\rm (i)]
$A_i E^*_0 E_0 = E^*_i E_0$;
\item[\rm (ii)]
$A^*_i E_0 E^*_0 = E_i E^*_0$;
\item[\rm (iii)]
$E_0 E^*_0 A_i = E_0 E^*_i$;
\item[\rm (iv)]
$E^*_0 E_0 A^*_i = E^*_0 E_i$.
\end{itemize}
\end{lemma}

\begin{proof}
(i)
Use Lemmas \ref{lem:rho}, \ref{lem:rhorhos}.

(ii) 
Apply (i) to $\Phi^*$.

(iii), (iv)
For the equations in (i) and (ii), apply $\dagger$ to each side and use Lemma \ref{lem:dagger4}.
\end{proof}

\begin{lemma}     \label{lem:A0As0}    \samepage
\ifDRAFT {\rm lem:A0As0}. \fi
We have $A_0 = I$.
\end{lemma}

\begin{proof}
By Lemma \ref{lem:rhoI}, $I^\rho = E^*_0$.
In this equation, apply $\rho^*$ to each side and evaluate the result using Lemma \ref{lem:rhorhospre}
and Definition \ref{def:Ai}.
\end{proof}

\begin{lemma}   \label{lem:sumAi}    \samepage
\ifDRAFT {\rm lem:sumAi}. \fi
We have
$\sum_{i=0}^d A_i = \nu E_0$.
\end{lemma}

\begin{proof}
In the equation  $\sum_{i=0}^d E^*_i = I$, apply $\rho^*$ to each side
and evaluate the result using Definition \ref{def:Ai} along with Lemma \ref{lem:rhoI}
applied to $\Phi^*$. 
\end{proof}

\begin{lemma}    \label{lem:AiAsi}    \samepage
\ifDRAFT {\rm lem:AiAsi}. \fi
The elements $\{A_i\}_{i=0}^d$ form a basis of the vector space $\cal M$.
\end{lemma}

\begin{proof}
By Lemmas \ref{lem:rhorhospre}, \ref{lem:rhorhos}
and since $\{E^*_i\}_{i=0}^d$ form a basis of  the vector space ${\cal M}^*$.
\end{proof}

\section{The scalars $k_i$}
\label{sec:ki}

We continue to discuss a symmetric idempotent system
$\Phi = (\{E_i\}_{i=0}^d; \{E^*_i\}_{i=0}^d)$ on $V$.
In this section we use $\Phi$ to define some scalars $k_i$ that will play a role 
in our theory.

\begin{definition}     \label{def:ki}    \samepage
\ifDRAFT {\rm def:ki}. \fi
For $0 \leq i \leq d$ let $k_i$ denote the eigenvalue of $A_i$ 
corresponding to $E_0$.
\end{definition}

\begin{lemma}   \label{lem:AiE0}    \samepage
\ifDRAFT {\rm lem:AiE0}. \fi
For $0 \leq i \leq d$ the following hold:
\begin{itemize}
\item[\rm (i)]
$A_i E_0  = E_0 A_i = k_i E_0$;
\item[\rm (ii)]
$A^*_i E^*_0 = E^*_0 A^*_i = k^*_i E^*_0$.
\end{itemize}
\end{lemma}

\begin{proof}
(i)
By Definition \ref{def:ki}.

(ii) 
Apply (i) to $\Phi^*$.
\end{proof}

\begin{lemma}    \label{lem:kimi}    \samepage
\ifDRAFT {\rm lem:kimi}. \fi
For $0 \leq i \leq d$ the following hold:
\begin{itemize}
\item[\rm (i)]
$k_i = \nu m^*_i$;
\item[\rm (ii)]
$k^*_i = \nu m_i$.
\end{itemize}
\end{lemma}

\begin{proof}
(i)
By Lemma \ref{lem:AiEs0E0}(i),
$E_0 A_i E^*_0 E_0 = E_0 E^*_i E_0$.
In this equation, evaluate the left-hand side using Lemmas \ref{lem:nuE0Es0E0}, \ref{lem:AiE0}(i),
and evaluate the right-hand side using Lemma \ref{lem:E0EsiE0}(ii).
This gives $k_i \nu^{-1} E_0 = m^*_i E_0$.
The result follows. 

(ii)
Apply (i) to $\Phi^*$.
\end{proof}

\begin{lemma}   \label{lem:ki}    \samepage
\ifDRAFT {\rm lem:ki}. \fi
The following hold:
\begin{itemize}
\item[\rm (i)]
$k_i \neq 0 \qquad (0 \leq i \leq d)$;
\item[\rm (ii)]
$\nu = \sum_{i=0}^d k_i$;
\item[\rm (iii)]
$k_0 = 1$.
\end{itemize}
\end{lemma}

\begin{proof}
(i)
Apply Lemma \ref{lem:mi3}(i) to $\Phi^*$ and use Lemma \ref{lem:kimi}(i).

(ii)
Apply Lemma \ref{lem:mi3}(ii) to $\Phi^*$ and use Lemma \ref{lem:kimi}(i).

(iii)
By Definition \ref{def:nu} and Lemma \ref{lem:kimi}(i).
\end{proof}

\section{Some reduction rules}
\label{sec:reduction}

We continue to discuss a symmetric idempotent system
$\Phi = (\{E_i\}_{i=0}^d; \{E^*_i\}_{i=0}^d)$ on $V$.
In this section we obtain some reduction rules for $\Phi$.
Recall the antiautomorphism $\dagger$ of $\text{\rm End}(V)$ from Definition \ref{def:sym}.

\begin{lemma}    \label{lem:EiEs0E0}     \samepage
\ifDRAFT {\rm lem:EiEs0E0}. \fi
For $0 \leq i \leq d$ the following hold:
\begin{itemize}
\item[\rm (i)]
$E_i E^*_0 E_0 = \nu^{-1} A^*_i E_0$;
\item[\rm (ii)]
$E^*_i E_0 E^*_0 = \nu^{-1} A_i E^*_0$;
\item[\rm (iii)]
$E_0 E^*_0 E_i = \nu^{-1} E_0 A^*_i$;
\item[\rm (iv)]
$E^*_0 E_0 E^*_i = \nu^{-1} E^*_0 A_i$.
\end{itemize}
\end{lemma}

\begin{proof}
(i)
Set $Y=E_i$ in \eqref{eq:defrho} and use Lemma \ref{lem:rhorhos}.

(ii) 
Apply (i) to $\Phi^*$.

(iii), (iv)
For the equations in  (i) and (ii), apply $\dagger$ to each side.
\end{proof}

\begin{lemma}   \label{lem:EsjAiEs0}     \samepage
\ifDRAFT {\rm lem:EsjAiEs0}. \fi
For $0 \leq i,j \leq d$ the following hold:
\begin{itemize}
\item[\rm (i)]
$E^*_j A_i E^*_0 = \delta_{i,j} A_i E^*_0$;
\item[\rm (ii)]
$E_j A^*_i E_0 = \delta_{i,j} A^*_i E_0$;
\item[\rm (iii)]
$E^*_0 A_i E^*_j = \delta_{i,j} E^*_0 A_i$;
\item[\rm (iv)]
$E_0 A^*_i E_j = \delta_{i,j} E_0 A^*_i$.
\end{itemize}
\end{lemma}

\begin{proof}
(i)
For the equation in Lemma \ref{lem:EiEs0E0}(ii),
multiply each side on the left by $E^*_j$ to get
$\delta_{i,j} E^*_i E_0 E^*_0 = \nu^{-1} E^*_j A_i E^*_0$.
In this equation, evaluate the left-hand side using Lemma \ref{lem:EiEs0E0}(ii).

(ii)
Apply (i) to $\Phi^*$.

(iii), (iv)
For the equations in  (i) and (ii), apply $\dagger$ to each side.
\end{proof}

\begin{lemma}   \label{lem:E0EsjAiEs0}     \samepage
\ifDRAFT {\rm lem:E0EsjAiEs0}. \fi
For $0 \leq i,j \leq d$ the following hold:
\begin{itemize}
\item[\rm (i)]
$E_0 E^*_j A_i E^*_0 = \delta_{i,j} k_i E_0 E^*_0$;
\item[\rm (ii)]
$E^*_0 E_j A^*_i E_0 = \delta_{i,j} k^*_i E^*_0 E_0$;
\item[\rm (iii)]
$E^*_0 A_i E^*_j E_0 = \delta_{i,j} k_i E^*_0 E_0$;
\item[\rm (iv)]
$E_0 A^*_i E_j E^*_0 = \delta_{i,j} k^*_i E_0 E^*_0$.
\end{itemize}
\end{lemma}

\begin{proof}
(i)
Using Lemmas \ref{lem:EsjAiEs0}(i), \ref{lem:AiE0}(i) in order,
\[
  E_0 E^*_j A_i E_0 = \delta_{i,j} E_0 A_i E^*_0 = \delta_{i,j} k_i E_0 E^*_0.
\]

(ii)
Apply (i) to $\Phi^*$.

(iii), (iv)
For the equations in  (i) and (ii), apply $\dagger$ to each side.
\end{proof}

\begin{lemma}    \label{lem:AiEs0Aj}     \samepage
\ifDRAFT {\rm lem:AiEs0Aj}. \fi
For $0 \leq i,j \leq d$ the following hold:
\begin{itemize}
\item[\rm (i)]
$A_i E^*_0 A_j =\nu E^*_i E_0 E^*_j$;
\item[\rm (ii)]
$A^*_i E_0 A^*_j = \nu E_i E^*_0 E_j$.
\end{itemize}
\end{lemma}

\begin{proof}
(i)
Using Lemmas \ref{lem:EiEs0E0}(iv), \ref{lem:AiEs0E0}(i) in order,
\[
  A_i E^*_0 A_j = \nu A_i E^*_0 E_0 E^*_j
 = \nu E^*_i E_0 E^*_j.
\]

(ii)
Apply (i) to $\Phi^*$.
\end{proof}

\begin{lemma}    \label{lem:EiEs0Aj}    \samepage
\ifDRAFT {\rm lem:EiEs0Aj}. \fi
For $0 \leq i,j \leq d$ the following hold:
\begin{itemize}
\item[\rm (i)]
$E_i E^*_0 A_j = A^*_i E_0 E^*_j$;
\item[\rm (ii)]
$E^*_i E_0 A^*_j = A_i E^*_0 E_j$;
\item[\rm (iii)]
$A_j E^*_0 E_i = E^*_j E_0 A^*_i$;
\item[\rm (iv)]
$A^*_j E_0 E^*_i = E_j E^*_0 A_i$.
\end{itemize}
\end{lemma}

\begin{proof}
(i)
Using Lemmas \ref{lem:EiEs0E0}(iv), \ref{lem:EiEs0E0}(i) in order,
\[
E_i E^*_0 A_j = \nu E_i E^*_0 E_0 E^*_j
  = A^*_i E_0 E^*_j.
\]

(ii)
Apply (i) to $\Phi^*$.

(iii), (iv)
For the equations in  (i) and (ii), apply $\dagger$ to each side.
\end{proof}

\section{The scalars $p^h_{i j}$, $q^h_{i j}$}
\label{sec:phij}

We continue to discuss a symmetric idempotent system
$\Phi = (\{E_i\}_{i=0}^d; \{E^*_i\}_{i=0}^d)$ on $V$.

\begin{lemma}   \label{lem:defp}    \samepage
\ifDRAFT {\rm lem:defp}. \fi
There exist scalars $p^h_{i j}$ $(0 \leq h,i,j \leq d)$ such that
\begin{align}
  A_i A_j &= \sum_{h=0}^d p^h_{i j} A_h   &&   (0 \leq i,j \leq d).       \label{eq:AiAj}
\end{align}
\end{lemma}

\begin{proof}
By Lemma \ref{lem:AiAsi}.
\end{proof}

\begin{definition}   \label{def:intersectionnumber}    \samepage
\ifDRAFT {\rm def:intersectionnumber}. \fi
Referring to Lemma \ref{lem:defp}, the scalars $p^h_{i j}$ 
are called the {\em intersection numbers} of $\Phi$.
\end{definition}

\begin{definition}     \label{def:qhij}     \samepage
\ifDRAFT {\rm def:qhij}. \fi
For $0 \leq h,i,j \leq d$ define  $q^h_{i j} = (p^h_{i j})^*$.
We call these scalars the {\em Krein parameters} of $\Phi$.
\end{definition}

\begin{lemma}    \label{lem:AsiAsj}    \samepage
\ifDRAFT {\rm lem:AsiAsj}. \fi
For $0 \leq i,j \leq d$, 
\begin{align}
  A^*_i A^*_j &= \sum_{h=0}^d q^h_{i j} A^*_h.       \label{eq:AsiAsj}
\end{align}
\end{lemma}

\begin{proof}
Apply Lemma \ref{lem:defp} to $\Phi^*$ and use Definition \ref{def:qhij}.
\end{proof}

\begin{lemma}    \label{lem:phijphji}    \samepage
\ifDRAFT {\rm lem:phijphji}. \fi
For $0 \leq h,i,j \leq d$ the following hold:
\begin{itemize}
\item[\rm (i)]
$p^h_{i j} = p^h_{j i}$;
\item[\rm (ii)]
$q^h_{i j} = q^h_{j i}$.
\end{itemize}
\end{lemma}

\begin{proof}
(i) 
By \eqref{eq:AiAj} and since the algebra $\cal M$ is commutative.

(ii) Apply (i) to $\Phi^*$.
\end{proof}

\begin{lemma}    \label{lem:phi0}    \samepage
\ifDRAFT {\rm lem:phi0}. \fi
For $0 \leq h,i \leq d$ the following hold:
\begin{itemize}
\item[\rm (i)]
$p^h_{i 0} = \delta_{h,i}$;
\item[\rm (ii)]
$p^h_{0 i} = \delta_{h,i}$;
\item[\rm (iii)]
$q^h_{i 0} = \delta_{h,i}$;
\item[\rm (iv)]
$q^h_{0 i} = \delta_{h,i}$.
\end{itemize}
\end{lemma}

\begin{proof}
(i)
In \eqref{eq:AiAj} set $j=0$ and use Lemmas \ref{lem:A0As0}, \ref{lem:AiAsi}.

(ii) By (i) and Lemma \ref{lem:phijphji}(i).

(iii), (iv)
Apply (i), (ii) to $\Phi^*$.
\end{proof}

\begin{lemma}   \label{lem:sumpthr}    \samepage
\ifDRAFT {\rm lem:sumpthr}. \fi
For $0 \leq h,i,j,t \leq d$ the following hold:
\begin{itemize}
\item[\rm (i)]
$\sum_{r=0}^d p^t_{h r} p^r_{i j} = \sum_{s=0}^d p^s_{h i} p^t_{s j}$;
\item[\rm (ii)]
$\sum_{r=0}^d q^t_{h r} q^r_{i j} = \sum_{s=0}^d q^s_{h i} q^t_{s j}$.
\end{itemize}
\end{lemma}

\begin{proof}
(i)
Expand $A_h (A_i A_j) = (A_h A_i) A_j$ in two ways using \eqref{eq:AiAj},
and compare the coefficients using Lemma \ref{lem:AiAsi}.

(ii) Apply (i) to $\Phi^*$.
\end{proof}

\begin{lemma}    \label{lem:sumphij}    \samepage
\ifDRAFT {\rm lem:sumphij}. \fi
For $0 \leq h,i \leq d$ the following hold:
\begin{itemize}
\item[\rm (i)]
$k_i = \sum_{j=0}^d p^h_{i j}$;
\item[\rm (ii)]
$k^*_i = \sum_{j=0}^d q^h_{i j}$.
\end{itemize}
\end{lemma}

\begin{proof}
(i)
Using Lemmas \ref{lem:sumAi} and \ref{lem:AiE0}(i),
\[
  A_i \sum_{j=0}^d A_j = k_i \sum_{h=0}^d A_h.
\]
By \eqref{eq:AiAj},
\[
 A_i \sum_{j=0}^d A_j = \sum_{h=0}^d \sum_{j=0}^d p^h_{i j} A_h.
\]
Compare the above two equations using Lemma \ref{lem:AiAsi}.

(ii)
Apply (i) to $\Phi^*$.
\end{proof}

\begin{lemma}   \label{lem:kip0ii}    \samepage
\ifDRAFT {\rm lem:kip0ii}. \fi
For $0 \leq i,j \leq d$ the following hold:
\begin{itemize}
\item[\rm (i)]
$p^0_{i j} = \delta_{i,j} k_i$;
\item[\rm (ii)]
$q^0_{i j} = \delta_{i,j} k^*_i$.
\end{itemize}
\end{lemma}

\begin{proof}
(i)
For the equation \eqref{eq:AiAj},  multiply each side on the left by $E_0 E^*_0$ and on the right
by $E^*_0 E_0$. 
Evaluate the result using  Lemma \ref{lem:AiEs0E0}(i),(iii)
along with Lemmas \ref{lem:E0EsiE0}, \ref{lem:kimi}.

(ii)
Apply (i) to $\Phi^*$.
\end{proof}

\begin{lemma}    \label{lem:kikj}    \samepage
\ifDRAFT {\rm lem:kikj}. \fi
For $0 \leq i,j \leq d$ the following hold:
\begin{itemize}
\item[\rm (i)]
$k_i k_j = \sum_{h=0}^d p^h_{i j} k_h$;
\item[\rm (ii)]
$k^*_i k^*_j = \sum_{h=0}^d q^h_{i j} k^*_h$.
\end{itemize}
\end{lemma}

\begin{proof}
(i)
In \eqref{eq:AiAj}, multiply each side by $E_0$, and simplify the result
using Lemma \ref{lem:AiE0}(i).

(ii)
Apply (i) to $\Phi^*$.
\end{proof}

\begin{lemma}    \label{lem:khphij}    \samepage
\ifDRAFT {\rm lem:khphij}. \fi
For $0 \leq h,i,j \leq d$ the following hold:
\begin{itemize}
\item[\rm (i)]
  $k_h p^h_{i j} = k_i p^i_{j h} = k_j p^j_{h i}$;
\item[\rm (ii)]
  $k^*_h q^h_{i j} = k^*_i q^i_{j h} = k^*_j q^j_{h i}$.
\end{itemize}
\end{lemma}

\begin{proof}
(i)
In view of Lemma \ref{lem:phijphji}(i),
it suffices to show that $k_h p^h_{i j} = k_j p^j_{h i}$.
To obtain this equation,
set $t=0$ in Lemma \ref{lem:sumpthr}(i),
and evaluate the result using Lemma \ref{lem:kip0ii}(i).

(ii)
Apply (i) to $\Phi^*$.
\end{proof}

\section{Reduction rules involving $p^h_{i j}$, $q^h_{i j}$}
\label{sec:red2}

We continue to discuss a symmetric idempotent system
$\Phi = (\{E_i\}_{i=0}^d; \{E^*_i\}_{i=0}^d)$ on $V$.
In this section we give some reduction rules for $\Phi$
that involve the intersection numbers and Krein parameters.

\begin{lemma}    \label{lem:AjEsiE0}    \samepage
\ifDRAFT {\rm lem:AjEsiE0}. \fi
For $0 \leq i,j \leq d$ the following hold:
\begin{itemize}
\item[\rm (i)]
$A_j E^*_i E_0 =  \sum_{h=0}^d p^h_{i j} E^*_h E_0$;
\item[\rm (ii)]
$A^*_j E_i E^*_0 = \sum_{h=0}^d q^h_{i j} E_h E^*_0$;
\item[\rm (iii)]
$E_0 E^*_i A_j = \sum_{h=0}^d p^h_{i j} E_0 E^*_h$;
\item[\rm (iv)]
$E^*_0 E_i A^*_j = \sum_{h=0}^d q^h_{i j} E^*_0 E_h$.
\end{itemize}
\end{lemma}

\begin{proof}
(i)
Using Lemmas \ref{lem:AiEs0E0}(i), \ref{lem:defp}, \ref{lem:AiEs0E0}(i) in order
\[
  A_j E^*_i E_0 = A_j A_i E^*_0 E_0
  = \sum_{h=0}^d p^h_{i j} A_h E^*_0 E_0
  = \sum_{h=0}^d p^h_{i j} E^*_h E_0.
\]

(ii)
Apply (i) to $\Phi^*$.

(iii), (iv)
For the equations in  (i) and (ii), apply $\dagger$ to each side.
\end{proof}

\begin{lemma}   \label{lem:E0EsiAjEsh}    \samepage
\ifDRAFT {\rm lem:E0EsiAjEsh}. \fi
For $0 \leq h,i,j \leq d$ the following hold:
\begin{itemize}
\item[\rm (i)]
$E^*_h A_j E^*_i E_0 = p^h_{i j} E^*_h E_0$;
\item[\rm (ii)]
$E_h A^*_j E_i E^*_0 = q^h_{i j} E_h E^*_0$;
\item[\rm (iii)]
$E_0 E^*_i A_j E^*_h = p^h_{i j} E_0 E^*_h$;
\item[\rm (iv)]
$E^*_0 E_i A^*_j E_h = q^h_{i j} E^*_0 E_h$.
\end{itemize}
\end{lemma}

\begin{proof}
(i)
Using Lemma \ref{lem:AjEsiE0}(i),
\[
E^*_h A_j E^*_i E_0
 = \sum_{s=0}^d p^s_{i j} E^*_h E^*_s E_0
 = p^h_{i j} E^*_h E_0.
\]

(ii)
Apply (i) to $\Phi^*$.

(iii), (iv)
For the equations in  (i) and (ii), apply $\dagger$ to each side.
\end{proof}

\begin{lemma}    \label{lem:EsiAjEsh}    \samepage
\ifDRAFT {\rm lem:EsiAjEsh}. \fi
For $0 \leq h,i,j \leq d$ the following hold:
\begin{itemize}
\item[\rm (i)]
$E_i A^*_j E_h = m_i^{-1} q^h_{i j} E_i E^*_0 E_h$;
\item[\rm (ii)]
$E^*_i A_j E^*_h = (m^*_i)^{-1} p^h_{i j} E^*_i E_0 E^*_h$.
\end{itemize}
\end{lemma}

\begin{proof}
(i)
In Lemma \ref{lem:E0EsiAjEsh}(iv), multiply each side on the left by $E_i$.
Simplify the result using Lemma \ref{lem:EsiE0Esi}(i).

(ii)
Apply (i) to $\Phi^*$.
\end{proof}

\begin{lemma}   \label{lem:EsiAjEsh2}    \samepage
\ifDRAFT {\rm lem:EsiAjEsh2}. \fi
For $0 \leq h,i,j \leq d$ the following hold:
\begin{itemize}
\item[\rm (i)]
$E^*_i A_j E^*_h = 0 \;\;$ if and only if $\;\; p^h_{i j} = 0$;
\item[\rm (ii)]
$E_i A^*_j E_h = 0  \;\;$ if and only if $\;\; q^h_{i j} = 0$.
\end{itemize}
\end{lemma}

\begin{proof}
By Lemmas \ref{lem:EsiE0Esj} and \ref{lem:EsiAjEsh}.
\end{proof}

\begin{lemma}  \label{lem:phij}    \samepage
\ifDRAFT {\rm lem:phij}. \fi
For $0 \leq h,i,j \leq d$ the following hold:
\begin{itemize}
\item[\rm (i)]
$p^h_{i j} = (m^*_h)^{-1} \text{\rm tr} (E_0 E^*_i A_j E^*_h)$;
\item[\rm (ii)]
$q^h_{i j} = m_h^{-1} \text{\rm tr} (E^*_0 E_i A^*_j E_h)$.
\end{itemize}
\end{lemma}

\begin{proof}
(i)
In Lemma \ref{lem:E0EsiAjEsh}(iii), take the trace of each side,
and simplify the result using Definition \ref{def:mi2}.

(ii)
Apply (i) to $\Phi^*$.
\end{proof}

\section{The scalars $p_i(j)$, $q_i(j)$}
\label{sec:pij}

We continue to discuss a symmetric idempotent system
$\Phi = (\{E_i\}_{i=0}^d; \{E^*_i\}_{i=0}^d)$ on $V$.
In this section we use $\Phi$ to define some scalars $p_i (j)$, $q_i (j)$
that will play a role in our theory.
Recall the algebra $\cal M$ from Definition \ref{def:D}.

\begin{lemma}    \label{lem:pij}    \samepage
\ifDRAFT {\rm lem:pij}. \fi
There exist scalars $p_i(j)$ $(0 \leq i,j \leq d)$  such that
\begin{align}
A_i &= \sum_{j=0}^d p_i(j) E_j  &&  (0 \leq i \leq d).    \label{eq:Ai}
\end{align}
\end{lemma}

\begin{proof}
By Definition \ref{def:D} the elements
$\{E_i\}_{i=0}^d$ form a basis of $\cal M$.
By Definition \ref{def:Ai}, $A_i \in {\cal M}$ for $0 \leq i \leq d$.
The result follows.
\end{proof}

\begin{definition}    \label{def:qij}    \samepage
\ifDRAFT {\rm def:qij}. \fi
For $0 \leq i,j \leq d$ define $q_i(j) = (p_i(j))^*$.
\end{definition}

\begin{lemma}    \label{lem:Asi}    \samepage
\ifDRAFT {\rm lem:Asi}. \fi
For $0 \leq i,j \leq d$,
\begin{align}
 A^*_i &= \sum_{j=0}^d q_i (j) E^*_j  &&  (0 \leq i \leq d).   \label{eq:Asi}
\end{align}
\end{lemma}

\begin{proof}
Apply Lemma \ref{lem:pij} to $\Phi^*$ and use Definition \ref{def:qij}.
\end{proof}

\begin{lemma}    \label{lem:AiEj}    \samepage
\ifDRAFT {\rm lem:AiEj}. \fi
For $0 \leq i,j \leq d$ the following hold:
\begin{itemize}
\item[\rm (i)]
$A_i E_j = E_j A_i = p_i (j) E_j$;
\item[\rm (ii)]
$A^*_i E^*_j = E^*_j A^*_i = q_i(j) E^*_j$.
\end{itemize}
In other words, $p_i (j)$ (resp.\ $q_i (j)$) is the eigenvalue of $A_i$
(resp.\ $A^*_i$) associated with $E_j V$ (resp.\ $E^*_j V$).
\end{lemma}

\begin{proof}
(i)
Use \eqref{eq:Ai}.

(ii)
Apply (i) to $\Phi^*$.
\end{proof}

\begin{lemma}    \label{lem:EiEsi}    \samepage
\ifDRAFT {\rm lem:EiEsi}. \fi
For $0 \leq i \leq d$ the following hold:
\begin{itemize}
\item[\rm (i)]
$E^*_i = \nu^{-1} \sum_{j=0}^d p_i (j) A^*_j$;
\item[\rm (ii)]
$E_i = \nu^{-1} \sum_{j=0}^d q_i(j) A_j$.
\end{itemize}
\end{lemma}

\begin{proof}
(i)
In \eqref{eq:Ai}, apply $\rho$ to each side and use
Lemma \ref{lem:rhorhos}.

(ii)
Apply (i) to $\Phi^*$.
\end{proof}

\begin{lemma}    \label{lem:pihqhj}    \samepage
\ifDRAFT {\rm lem:pihqhj}. \fi
For $0 \leq i,j \leq d$ the following hold:
\begin{itemize}
\item[\rm (i)]
$\sum_{h=0}^d p_i (h) q_h (j) = \delta_{i,j} \nu$;
\item[\rm (ii)]
$\sum_{h=0}^d q_i (h) p_h (j) = \delta_{i,j} \nu$.
\end{itemize}
\end{lemma}

\begin{proof}
(i)
By \eqref{eq:Ai}, $A_i = \sum_{h=0}^d p_i (h) E_h$.
In this equation, eliminate $E_h$ using Lemma \ref{lem:EiEsi}(ii),
and compare the coefficients of each side.

(ii)
Apply (i) to $\Phi^*$.
\end{proof}

\begin{lemma}  \label{lem:p0j}    \samepage
\ifDRAFT {\rm lem:p0j}. \fi
For $0 \leq j \leq d$ the following hold:
\begin{itemize}
\item[\rm (i)]
$p_0 (j) = 1$;
\item[\rm (ii)]
$q_0 (j) = 1$.
\end{itemize}
\end{lemma}

\begin{proof}
(i)
Set $i=0$ in \eqref{eq:Ai} and recall that $A_0 = I$.

(ii) Apply (i) to $\Phi^*$.
\end{proof}

\begin{lemma}   \label{lem:pi0}    \samepage
\ifDRAFT {\rm lem:pi0}. \fi
For $0 \leq i \leq d$ the following hold:
\begin{itemize}
\item[\rm (i)]
$p_i (0) = k_i$;
\item[\rm (ii)]
$q_i (0) = k^*_i$.
\end{itemize}
\end{lemma}

\begin{proof}
(i)
Set $j=0$ in Lemma \ref{lem:AiEj}(i) 
and compare the result with Lemma \ref{lem:AiE0}(i).

(ii)
Apply (i) to $\Phi^*$.
\end{proof}

\begin{lemma}   \label{lem:sumphj}   \samepage
\ifDRAFT {\rm lem:sumphj}. \fi
For $0 \leq j \leq d$ the following hold:
\begin{itemize}
\item[\rm (i)]
$\sum_{h=0}^d p_h (j) = \delta_{0, j} \nu$;
\item[\rm (ii)]
$\sum_{h=0}^d q_h (j) = \delta_{0,j} \nu$.
\end{itemize}
\end{lemma}

\begin{proof}
(i)
Set $i=0$ in Lemma \ref{lem:pihqhj}(ii), and evaluate the result
using Lemma \ref{lem:p0j}(ii).

(ii)
Apply (i) to $\Phi^*$.
\end{proof}

\begin{lemma}    \label{lem:kshpih}    \samepage
\ifDRAFT {\rm lem:kshpih}. \fi
For $0 \leq i \leq d$ the following hold:
\begin{itemize}
\item[\rm (i)]
$\sum_{h=0}^d m_h p_i (h) = \delta_{i,0}$;
\item[\rm (ii)]
$\sum_{h=0}^d m^*_h q_i (h) = \delta_{i,0}$.
\end{itemize}
\end{lemma}

\begin{proof}
(i)
Set $j=0$ in Lemma \ref{lem:pihqhj}(i), and evaluate the result
using Lemmas \ref{lem:kimi}(ii), \ref{lem:pi0}(ii).

(ii)
Apply (i) to $\Phi^*$.
\end{proof}

\begin{lemma}    \label{lem:pirpjr}    \samepage
\ifDRAFT {\rm lem:pirpjr}. \fi
For $0 \leq i,j,r \leq d$ the following hold:
\begin{itemize}
\item[\rm (i)]
$p_i (r) p_j (r) = \sum_{h=0}^d p^h_{i j} p_h (r)$;
\item[\rm (ii)]
$q_i (r) q_j (r) = \sum_{h=0}^d q^h_{i j} q_h (r)$.
\end{itemize}
\end{lemma}

\begin{proof}
(i)
In \eqref{eq:AiAj}, multiply each side by $E_r$,
and simplify the result using Lemma \ref{lem:AiEj}(i).

(ii)
Apply (i) to $\Phi^*$.
\end{proof}

\begin{lemma}    \label{lem:phijsum}    \samepage
\ifDRAFT {\rm lem:phijsum}. \fi
For $0 \leq h,i,j \leq d$ the following hold:
\begin{itemize}
\item[\rm (i)]
$p^h_{i j} = \nu^{-1} \sum_{r=0}^d p_i (r) p_j (r) q_r (h)$;
\item[\rm (ii)]
$q^h_{i j} = \nu^{-1} \sum_{r=0}^d q_i (r) q_j (r) p_r (h)$.
\end{itemize}
\end{lemma}

\begin{proof}
(i)
Expand the sum $\sum_{r=0}^d p_i (r) p_j (r) q_r (h)$
using Lemma \ref{lem:pirpjr}(i),
and simplify the result using Lemma \ref{lem:pihqhj}(i).

(ii)
Apply (i) to $\Phi^*$.
\end{proof}

\section{Reduction rules involving $p_i (j)$, $q_i (j)$}
\label{sec:red3}

We continue to discuss a symmetric idempotent system
$\Phi = (\{E_i\}_{i=0}^d; \{E^*_i\}_{i=0}^d)$ on $V$.

\begin{lemma}    \label{lem:E0AsiAj}     \samepage
\ifDRAFT {\rm lem:E0AsiAj}. \fi
For $0 \leq i,j \leq d$ the following hold:
\begin{itemize}
\item[\rm (i)]
$E_0 A^*_i A_j = p_j (i) E_0 A^*_i$;
\item[\rm (ii)]
$E^*_0 A_i A^*_j = q_j (i) E^*_0 A_i$;
\item[\rm (iii)]
$A_j A^*_i E_0 = p_j (i) A^*_i E_0$;
\item[\rm (iv)]
$A^*_j A_i E^*_0 = q_j (i) A_i E^*_0$.
\end{itemize}
\end{lemma}

\begin{proof}
(i)
Using Lemmas \ref{lem:pij} and \ref{lem:EsjAiEs0}(iv) in order,
\[
E_0 A^*_i A_j
= \sum_{h=0}^d p_j (h) E_0 A^*_i E_h
= \sum_{h=0}^d p_j (h) \delta_{i,h} E_0 A^*_i
= p_j (i) E_0 A^*_i.
\]

(ii)
Apply (i) to $\Phi^*$.

(iii), (iv)
For the equations in  (i) and (ii), apply $\dagger$ to each side.
\end{proof}

\begin{lemma}    \label{lem:E0EsiEj}    \samepage
\ifDRAFT {\rm lem:E0EsiEj}. \fi
For $0 \leq i,j \leq d$ the following hold:
\begin{itemize}
\item[\rm (i)]
$E_0 E^*_i E_j = \nu^{-1} p_i (j) E_0 A^*_j$;
\item[\rm (ii)]
$E^*_0 E_i E^*_j = \nu^{-1} q_i (j) E^*_0 A_j$;
\item[\rm (iii)]
$E_j E^*_i E_0 = \nu^{-1} p_i (j) A^*_j E_0$;
\item[\rm (iv)]
$E^*_j E_i E^*_0 = \nu^{-1} q_i (j) A_j E^*_0$.
\end{itemize}
\end{lemma}

\begin{proof}
(i)
Using Lemmas \ref{lem:EiEsi}(i) and \ref{lem:EsjAiEs0}(iv) in order,
\[
E_0 E^*_i E_j
 = E_0 \left( \nu^{-1} \sum_{h=0}^d p_i (h) A^*_h \right) E_j
 = \nu^{-1} \sum_{h=0}^d p_i (h) \delta_{h,j} E_0 A^*_h
 = \nu^{-1} p_i (j) E_0 A^*_j.
\]

(ii)
Apply (i) to $\Phi^*$.

(iii), (iv)
For the equations in  (i) and (ii), apply $\dagger$ to each side.
\end{proof}

\begin{lemma}    \label{lem:E0AsiAjEs0}    \samepage
\ifDRAFT {\rm lem:E0AsiAjEs0}. \fi
For $0 \leq i,j \leq d$ the following hold:
\begin{itemize}
\item[\rm (i)]
$E_0 A^*_i A_j E^*_0 = p_j (i) k^*_i E_0 E^*_0$;
\item[\rm (ii)]
$E^*_0 A_i A^*_j E_0 = q_j (i) k_i E^*_0 E_0$;
\item[\rm (iii)]
$E^*_0 A_i A^*_j E_0 = p_i (j) k^*_j E^*_0 E_0$;
\item[\rm (iv)]
$E_0 A^*_i A_j E^*_0 = q_i (j) k_j E_0 E^*_0$.
\end{itemize}
\end{lemma}

\begin{proof}
(i)
Using Lemmas \ref{lem:E0AsiAj}(i), \ref{lem:AiEj}(ii), \ref{lem:pi0}(ii) in order,
\[
E_0 A^*_i A_j E^*_0
= p_j (i) E_0 A^*_i E^*_0
= p_j (i) q_i (0) E_0 E^*_0
= p_j (i) k^*_i E_0 E^*_0.
\]

(ii)
Apply (i) to $\Phi^*$.

(iii), (iv)
For the equations in  (i) and (ii), apply $\dagger$ to each side.
\end{proof}

\begin{lemma}    \label{lem:E0EsiEjEs0}    \samepage
\ifDRAFT {\rm lem:E0EsiEjEs0}. \fi
For $0 \leq i,j \leq d$ the following hold:
\begin{itemize}
\item[\rm (i)]
$E_0 E^*_i E_j E^*_0 =  p_i (j) m_j E_0 E^*_0$;
\item[\rm (ii)]
$E^*_0 E_i E^*_j E_0 =  q_i (j) m^*_j E^*_0 E_0$.
\end{itemize}
\end{lemma}

\begin{proof}
(i)
Using Lemmas \ref{lem:E0EsiEj}(i), \ref{lem:AiEj}(ii), \ref{lem:pi0}(ii) in order,
\[
E_0 E^*_i E_j E^*_0
= \nu^{-1} p_i (j) E_0 A^*_j E^*_0
= \nu^{-1} p_i (j) q_j (0) E_0 E^*_0
= \nu^{-1} p_i (j) k^*_j E_0 E^*_0.
\]
Now use Lemma \ref{lem:kimi}(ii).

(ii)
Apply (i) to $\Phi^*$.
\end{proof}

\begin{lemma}   \label{lem:pijksj}    \samepage
\ifDRAFT {\rm lem:pijksj}. \fi
For $0 \leq i,j \leq d$,
\begin{equation}
 \frac{ p_i (j) } { k_i } =  \frac{ q_j (i) } {k^*_j }.              \label{eq:pijksj}
\end{equation}
\end{lemma}

\begin{proof}
By Lemma \ref{lem:E0AsiAjEs0}(ii),(iii),
$p_i (j) k^*_j E_0 E^*_0 = q_j (i) k_i E_0 E^*_0$.
The result follows since $E_0 E^*_0 \neq 0$ by Definition \ref{def:ips}(iii).
\end{proof}

\begin{lemma}    \label{lem:pij2}    \samepage
\ifDRAFT {\rm lem:pij2}. \fi
For $0 \leq i,j \leq d$ the following hold:
\begin{itemize}
\item[\rm (i)]
$p_i (j) = \nu  m_j^{-1} \, \text{\rm tr} (E_0 E^*_i E_j E^*_0)$;
\item[\rm (ii)]
$p_i (j) = \nu  m_j^{-1} \, \text{\rm tr} (E^*_0 E_j E^*_i E_0)$;
\item[\rm (iii)]
$q_i (j) = \nu  (m^*_j)^{-1} \, \text{\rm tr} (E^*_0 E_i E^*_j E_0)$;
\item[\rm (iv)]
$q_i (j) = \nu  (m^*_j)^{-1} \, \text{\rm tr} (E_0 E^*_j E_i E^*_0)$.
\end{itemize}
\end{lemma}

\begin{proof}
(i)
Using Lemmas \ref{lem:E0EsiEjEs0}(i) and Definition \ref{def:nu},
\[
 \text{\rm tr} (E_0 E^*_i E_j E^*_0)
 =  p_i (j) m_j \text{\rm tr}(E_0 E^*_0)
 = \nu^{-1} p_i (j) m_j.
\]

(iii)
Apply (i) to $\Phi^*$.

(ii)
In (iii), exchange $i$, $j$, and use Lemmas \ref{lem:kimi}(i),  \ref{lem:pijksj}.

(iv)
Apply (ii) to $\Phi^*$.
\end{proof}

\section{Some matrices}
\label{sec:matrices}

We continue to discuss a symmetric idempotent system
$\Phi = (\{E_i\}_{i=0}^d; \{E^*_i\}_{i=0}^d)$ on $V$.
In the previous sections we used $\Phi$ to define several kinds of scalars,
and we described how these scalars are related.
In this section we express these relationships in matrix form.

\begin{definition}    \label{def:matrices}    \samepage
\ifDRAFT {\rm def:matrices}. \fi
Let $K$ (resp. $K^*$) denote the diagonal matrix in $\Mat_{d+1}(\F)$
that has $(i,i)$-entry $k_i$ (resp.\ $k^*_i$) for $0 \leq i \leq d$.
Let $P$ (resp.\ $Q$) denote the matrix in $\Mat_{d+1}(\F)$ that has
$(i,j)$-entry $p_j (i)$ (resp.\ $q_j (i)$) for $0 \leq i,j \leq d$.
\end{definition}

\begin{lemma}    \label{lem:PtKs}    \samepage
\ifDRAFT {\rm lem:PtKs}. \fi
The following hold:
\begin{itemize}
\item[\rm (i)]
$P Q = Q P  = \nu I$;
\item[\rm (ii)]
$P^{\sf t} K^* = K Q$;
\item[\rm (iii)]
$K^* P = Q^{\sf t} K$.
\end{itemize}
\end{lemma}

\begin{proof}
(i)
By Lemma \ref{lem:pihqhj}.

(ii) 
By Lemma \ref{lem:pijksj}.

(iii)
In (ii), take the transpose of each side. 
\end{proof}

\begin{definition}     \label{def:U}   \samepage
\ifDRAFT {\rm def:U}. \fi
Note by Lemma \ref{lem:PtKs} that
$K^{-1} P^{\sf t} = Q (K^*)^{-1}$ and
$(K^*)^{-1} Q^{\sf t} = P K^{-1}$;
we define
\begin{align}
 U &= K^{-1} P^{\sf t} = Q (K^*)^{-1},  &
 U^* &= (K^*)^{-1} Q^{\sf t} = P K^{-1}.    \label{eq:defU}
\end{align}
\end{definition}

\begin{lemma}    \label{lem:U3}    \samepage
\ifDRAFT {\rm lem:U3}. \fi
The following hold:
\begin{itemize}
\item[\rm (i)]
$P = U^* K$;
\item[\rm (ii)]
$P^{\sf t} = K U$;
\item[\rm (iii)]
$Q = U K^*$;
\item[\rm (iv)]
$Q^{\sf t} = K^* U^*$.
\end{itemize}
\end{lemma}

\begin{proof}
Immediate from Definition \ref{def:U}.
\end{proof}

\begin{lemma}    \label{lem:U0j}    \samepage
\ifDRAFT {\rm lem:U0j}. \fi
We have $U_{i,0}=1$ and $U^*_{i,0} =1$ for $0 \leq i \leq d$.
Moreover $U_{0,j} = 1$ and $U^*_{0,j} = 1$ for $0 \leq j \leq d$.
\end{lemma}

\begin{proof}
Use Lemmas \ref{lem:ki}(iii), \ref{lem:p0j}, \ref{lem:pi0}.
\end{proof}

\begin{lemma}    \label{lem:U2}    \samepage
\ifDRAFT {\rm lem:U2}. \fi
The following hold:
\begin{itemize}
\item[\rm (i)]
$U^{\sf t} = U^*$;
\item[\rm (ii)]
$U K^* U^* K = \nu I$;
\item[\rm (iii)]
$U^* K U K^* = \nu I$.
\end{itemize}
\end{lemma}

\begin{proof}
(i)
By Definition \ref{def:U}.

(ii), (iii)
By Lemma \ref{lem:U3}(i),(iii) and Lemma \ref{lem:PtKs}(i).
\end{proof}

\begin{definition}    \label{def:matrices2}    \samepage
\ifDRAFT {\rm def:matrices2}. \fi
For $0 \leq i \leq d$ let $B_i$ and  $B^*_i$ denote the matrices in $\Mat_{d+1}(\F)$
that have entries
\begin{align*}
  (B_i)_{h,j} &= p^h_{i j},  &
  (B^*_i)_{h,j} &= q^h_{i j}  &&  (0 \leq h,j \leq d).
\end{align*}
We call $B_i$ (resp.\ $B^*_i$) the {\em $i^\text{th}$ intersection matrix}
(resp.\ {\em $i^\text{th}$ dual intersection matrix}) of $\Phi$.
\end{definition}

\begin{definition}   \label{def:matrices3}    \samepage
\ifDRAFT {\rm def:matrices3}. \fi
For $0 \leq i \leq d$ let $H_i$ and $H^*_i$ denote the diagonal matrices in $\Mat_{d+1}(\F)$ 
that have diagonal entries
\begin{align*}
  (H_i)_{j,j} &= p_i (j),  &
  (H^*_i)_{j,j} &= q_i (j)   &&   (0 \leq j \leq d).
\end{align*}
\end{definition}

\begin{lemma}    \label{lem:BDPQ}    \samepage
\ifDRAFT {\rm lem:BDPQ}. \fi
For $0 \leq r \leq d$,
\begin{align}
H_r P &= P B_r,   &  H^*_r Q &= Q B^*_r,              \label{eq:DrP}
\\
Q H_r  &= B_r Q,  &  P H^*_r &= B^*_r P,              \label{eq:QDr}
\\
K B_r &= (B_r)^{\sf t} K, &  K^* B^*_r &= (B^*_r)^{\sf t} K^*,  \label{eq:KBr}
\\
U H_r &= B_r U,  &  U^* H^*_r &= B^*_r U^*.             \label{eq:UDr}
\end{align}
\end{lemma}

\begin{proof}
To get the equation on the left in \eqref{eq:DrP},
compare the entries of each side using Lemma \ref{lem:pirpjr}(i).
In the equation on the left in \eqref{eq:DrP}, multiply each side on the left and on the right by $Q$
and simplify the result using Lemma \ref{lem:PtKs}(i). 
This gives the equation on the left in \eqref{eq:QDr}.
To obtain the equation on the left in \eqref{eq:KBr}, 
compare the entries of each side using Lemma \ref{lem:khphij}(i).
The equation on the left in \eqref{eq:UDr} follows from $Q H_r = B_r Q$ and
Lemma \ref{lem:U3}(iii) together with the fact that $H_r$, $K^*$ commute
since they are both diagonal.
To get the equations on the right in \eqref{eq:DrP}--\eqref{eq:UDr},
apply the equations on the left in \eqref{eq:DrP}--\eqref{eq:UDr} to $\Phi^*$.
\end{proof}

\begin{lemma}    \label{lem:BrBs}    \samepage
\ifDRAFT {\rm lem:BrBs}. \fi
For $0 \leq i,j \leq d$ the following hold:
\begin{itemize}
\item[\rm (i)]
$B_i B_j = \sum_{h=0}^d p^h_{i j} B_h$;
\item[\rm (ii)]
$B^*_i B^*_j = \sum_{h=0}^d q^h_{i j} B^*_h$;
\item[\rm (iii)]
$H_i H_j = \sum_{h=0}^d p^h_{i j} H_h$;
\item[\rm (iv)]
$H^*_i H^*_j = \sum_{h=0}^d q^h_{i j} H^*_h$.
\end{itemize}
\end{lemma}

\begin{proof}
(i), (ii)
By Lemma \ref{lem:sumpthr}.

(iii), (iv)
By Lemma \ref{lem:pirpjr}.
\end{proof}

\section{The $\Phi$-standard basis}
\label{sec:standard}

We continue to discuss a symmetric idempotent system
$\Phi = (\{E_i\}_{i=0}^d; \{E^*_i\}_{i=0}^d)$ on $V$.
In this section we introduce the notion of a $\Phi$-standard basis.

\begin{lemma}    \label{lem:EsiE0V}    \samepage
\ifDRAFT {\rm lem:EsiE0V}. \fi
For $0 \leq i \leq d$, $E^*_i V = E^*_i E_0 V$.
\end{lemma}

\begin{proof}
The vector space $E^*_i V$ has dimension $1$ and contains $E^*_i E_0 V$.
By Definition \ref{def:ips}(iii), $E^*_i E_0 V \neq 0$.
The result follows.
\end{proof}

\begin{lemma}    \label{lem:Esiu}    \samepage
\ifDRAFT {\rm lem:Esiu}. \fi
Let $\xi$ denote a nonzero vector in $E_0 V$.
Then for $0 \leq i \leq d$ the vector $E^*_i \xi$ is nonzero and hence
a basis of $E^*_i V$.
Moreover the vectors $\{E^*_i \xi\}_{i=0}^d$ form a basis of $V$.
\end{lemma}

\begin{proof}
Let the integer $i$ be given.
We show $E^*_i \xi \neq 0$.
The vector space $E_0 V$ has dimension $1$ and $\xi$ is a nonzero vector in $E_0 V$, so
$\xi$ spans $E_0 V$.
Therefore $E^*_i \xi$ spans $E^*_i E_0 V$.
The vector space $E^*_i E_0 V$ has dimension $1$ by Lemma \ref{lem:EsiE0V}
so $E^*_i \xi$ is nonzero.
The remaining assertions are clear.
\end{proof}

\begin{definition}    \label{def:standard}    \samepage
\ifDRAFT {\rm def:standard}. \fi
By a {\em $\Phi$-standard basis} of $V$ we mean a sequence
$\{E^*_i \xi\}_{i=0}^d$, where $\xi$ is a nonzero vector in $E_0 V$.
\end{definition}

We give a characterization of a $\Phi$-standard basis.

\begin{lemma}    \label{lem:standard}    \samepage
\ifDRAFT {\rm lem:standard}. \fi
Let $\{u_i\}_{i=0}^d$ denote a sequence of vectors in $V$, not all $0$.
Then this sequence is a $\Phi$-standard basis if and only if both {\rm (i), (ii)}
hold below:
\begin{itemize}
\item[\rm (i)]
$u_i \in E^*_i V$ for $0 \leq i \leq d$;
\item[\rm (ii)]
$\sum_{i=0}^d u_i \in E_0 V$.
\end{itemize}
\end{lemma}

\begin{proof}
To prove the lemma in one direction,
assume that $\{u_i\}_{i=0}^d$ is a $\Phi$-standard basis of $V$.
By Definition \ref{def:standard} there exists a nonzero $\xi \in E_0 V$
such that $u_i = E^*_i \xi$ for $0 \leq i \leq d$.
By construction $u_i \in E^*_i V$ for $0 \leq i \leq d$, so (i) holds.
Recall $I = \sum_{i=0}^d E^*_i$.
In this equation we apply each side to $\xi$, to find that $\xi = \sum_{i=0}^d u_i$,
and (ii) follows.
We have now proved the lemma in one direction.
To prove the lemma in the other direction,
assume that $\{u_i\}_{i=0}^d$ satisfy (i) and (ii).
Define $\xi = \sum_{i=0}^d u_i$ and observe $\xi \in E_0 V$.
Using (i) we find that $E^*_i u_j = \delta_{i,j} u_i$ for $0 \leq i,j \leq d$.
It follows $u_i = E^*_i \xi$ for $0 \leq i \leq d$.
Observe $\xi \neq 0$ since at least one of $\{u_i\}_{i=0}^d$ is nonzero.
Now $\{u_i\}_{i=0}^d$ is a $\Phi$-standard basis of $V$ by Definition \ref{def:standard}.
\end{proof}

\section{Bilinear forms} 
\label{sec:bilin}

In this section we recall some basic facts concerning bilinear forms on $V$.
See \cite[Section 8.5]{Rot} for more information.
By a {\em bilinear form on $V$} we mean a map $\b{\; , \; } : V \times V \to \F$
that satisfies the following four conditions for $u,v,w \in V$ and $\alpha \in \F$:
(i) $\b{u+v,w} = \b{u,w} + \b{v,w}$;
(ii) $\b{\alpha u, v} = \alpha \b{u,v}$;
(iii) $\b{u,v+w} = \b{u,v} + \b{u,w}$;
(iv) $\b{u,\alpha v} = \alpha \b{u,v}$.
Let $\b{ \; , \; }$ denote a bilinear form on $V$.
We abbreviate $||v||^2 = \b{ v,v }$ for $v \in V$.
The following are equivalent:
(i) there exists a nonzero $u \in V$ such that
$\b{ u, v} = 0$ for all $v \in V$;
(ii) there exists a nonzero $v \in V$ such that
$\b{ u,v } = 0$ for all $u \in V$.
The form $\b{ \; , \; }$ is said to be {\em degenerate}
whenever (i), (ii) hold and {\em nondegenerate} otherwise.

Let $\gamma$ denote an antiautomorphism of $\text{\rm End}(V)$.
Then there exists  a nonzero bilinear form $\b{\; , \;}$ on $V$ such that
$\bbig{A u, v} = \bbig{u, A^\gamma v}$ for $u,v \in V$ and  $A \in \text{\rm End}(V)$.
The form is unique up to multiplication by a nonzero scalar.
The form is nondegenerate.
We refer to this form as a {\em bilinear form on $V$ associated with $\gamma$}.

For the rest of this section let $\b{ \; , \; }$ denote a nondegenerate bilinear form on $V$.

\begin{definition}     \label{def:innermatrix}    \samepage
\ifDRAFT {\rm def:innermatrix}. \fi
For bases $\{u_i\}_{i=0}^d$ and $\{v_i\}_{i=0}^d$ of $V$,
the {\em inner product matrix from $\{u_i\}_{i=0}^d$ to $\{v_i\}_{i=0}^d$} 
is the matrix in $\Mat_{d+1}(\F)$ that has $(i,j)$-entry $\b{ u_i, v_j }$
for $0 \leq i,j \leq d$.
\end{definition}

Referring to Definition \ref{def:innermatrix}, the inner product matrix from $\{u_i\}_{i=0}^d$
to $\{v_i\}_{i=0}^d$ is invertible.

\begin{definition}    \label{def:symmetric}    \samepage
\ifDRAFT {\rm def:symmetric}. \fi
The form $\b{\; , \;}$ is said to be {\em symmetric} whenever
$\b{u,v} = \b{v,u}$ for $u,v\in V$.
\end{definition}

\begin{definition}    \label{def:dualbasis}     \samepage
\ifDRAFT {\rm def:dualbasis}. \fi
Assume that $\b{ \; , \; }$ is symmetric.
Then two bases $\{u_i\}_{i=0}^d$, $\{v_i\}_{i=0}^d$ of $V$ are said to be 
{\em dual with respect to $\b{ \; , \; }$}
whenever $\b{u_i, v_j} = \delta_{i,j}$ for $0 \leq i,j \leq d$.
\end{definition}

\begin{lemma}      \label{lem:dualbasis}    \samepage
\ifDRAFT {\rm lem:dualbasis}. \fi
Assume that $\b{ \; , \; }$ is  symmetric.
Then each basis of $V$ has a unique dual with respect to $\b{ \; , \; }$.
\end{lemma}

\begin{lemma}    \label{lem:trans}    \samepage
\ifDRAFT {\rm lem:trnas}. \fi
Assume that $\b{ \; , \; }$ is  symmetric.
Let $\{u_i\}_{i=0}^d$ and $\{v_i\}_{i=0}^d$ denote bases of $V$.
Then the following are the same:
\begin{itemize}
\item[\rm (i)]
the inner product matrix from $\{u_i\}_{i=0}^d$ to $\{v_i\}_{i=0}^d$;
\item[\rm (ii)]
the inner product matrix from $\{u_i\}_{i=0}^d$ to $\{u_i\}_{i=0}^d$,
times the transition matrix from $\{u_i\}_{i=0}^d$ to $\{v_i\}_{i=0}^d$.
\end{itemize}
\end{lemma}

\begin{proof}
Routine linear algebra.
\end{proof}

\section{The dual $\Phi$-standard basis}
\label{sec:dualstandard}

We return our attention to a symmetric idempotent system
 $\Phi = (\{E_i\}_{i=0}^d; \{E^*_i\}_{i=0}^d)$ on $V$.
In this section we introduce the notion of a dual $\Phi$-standard basis of $V$.
Recall the antiautomorphism  $\dagger$ of $\text{\rm End}(V)$ from Definition \ref{def:sym}.
For the rest of the paper $\b{\; , \;}$ denotes a bilinear form on $V$
associated with $\dagger$.
By the construction, for $A \in \text{\rm End}(V)$ we have
\begin{align}
 \bbig{ A u, v} &= \bbig{u, A^\dagger v}  &&  (u,v \in V).   \label{eq:bXuv}
\end{align}
Recall the algebra $\cal M$ from Definition \ref{def:D}.

\begin{lemma}    \label{lem:bilin1}    \samepage
\ifDRAFT {\rm lem:bilin1}. \fi
For $A \in {\cal M} \cup {\cal M}^*$,
\begin{align}
  \bbig{ A u, v} &= \bbig{u, A v}  &&  (u,v \in V).        \label{eq:bXuv2}
\end{align}
\end{lemma}

\begin{proof}
By Definition \ref{def:sym} and \eqref{eq:bXuv}.
\end{proof}

\begin{lemma}     \label{lem:EsiwEsjw}    \samepage
\ifDRAFT {\rm lem:EsiwEsjw}. \fi
For $\xi \in E_0 V$,
\begin{align}
 \bbig{E^*_i \xi, E^*_j \xi} &= \delta_{i,j} \nu^{-1} k_i  ||\xi||^2
          &&                  (0 \leq i,j \leq d).                    \label{eq:EsiwEsjw}
\end{align}
\end{lemma}

\begin{proof}
Using \eqref{eq:bXuv2} and $E_0 \xi = \xi$,
\[
\bbig{ E^*_i \xi, E^*_j \xi } 
 = \bbig{ E^*_i E_0 \xi, E^*_j E_0 \xi }
 = \bbig{ \xi, E_0 E^*_i E^*_j E_0 \xi }
 = \delta_{i,j} \bbig{ \xi, E_0 E^*_i E_0 \xi}.
\]
By this and Lemmas \ref{lem:E0EsiE0}(ii), \ref{lem:kimi}(i) we get the result.
\end{proof}

\begin{lemma}    \label{lem:bsym}    \samepage
\ifDRAFT {\rm lem:sym}.  \fi
The bilinear form $\b{ \; , \;}$ is symmetric.
\end{lemma}

\begin{proof}
Consider a $\Phi$-standard basis $\{E^*_i \xi\}_{i=0}^d$ of $V$, where $0 \neq \xi \in E_0 V$.
By Lemma \ref{lem:EsiwEsjw}, 
$\bbig{E^*_i \xi, E^*_j \xi} = \bbig{ E^*_j \xi, E^*_i \xi}$ for $0 \leq i,j \leq d$.
Therefore $\b{u,v} = \b{v,u}$ for $u$, $v \in V$.
\end{proof}

\begin{lemma}    \label{lem:nonzero1}    \samepage
\ifDRAFT {\rm lem:nonzero1}. \fi
The following hold for $0 \neq \xi \in E_0 V$ and
$0 \neq \xi^* \in E^*_0 V$:
\begin{itemize}
\item[\rm (i)]
each of $||\xi||^2$, $||\xi^*||^2$, $\bbig{\xi, \xi^*}$ is nonzero;
\item[\rm (ii)]
$E^*_0 \xi = \frac{ \b{\xi, \xi^*} } { || \xi^* ||^2 } \, \xi^*$;
\item[\rm (iii)]
$E_0 \xi^* = \frac{ \b{\xi, \xi^*} } { || \xi ||^2 } \, \xi$;
\item[\rm (iv)]
$ || \xi ||^2 || \xi^* ||^2 = \nu \, \bbig{\xi, \xi^*}^2$.
\end{itemize}
\end{lemma}

\begin{proof}
(i)
Observe $||\xi||^2 \neq 0$ by Lemma \ref{lem:EsiwEsjw} and since $\b{ \; , \; }$ is nonzero.
Applying this to $\Phi^*$ we get $||\xi^*||^2 \neq 0$.
To see that $\bbig{\xi, \xi^*} \neq 0$,
observe that $\xi^*$ is a basis of $E^*_0 V$ so there exists a scalar $\alpha$ such that
$E^*_0 \xi = \alpha \xi^*$.
Recall $E^*_0 \xi \neq 0$ by Lemma \ref{lem:Esiu} so $\alpha \neq 0$.
Using \eqref{eq:bXuv2} and $E^*_0 \xi^* = \xi^*$ we routinely find that 
$\bbig{\xi, \xi^*} = \alpha ||\xi^*||^2$ and it follows $\bbig{\xi, \xi^*} \neq 0$.

(ii)
In the proof of part (i) we found $E^*_0 \xi = \alpha \xi^*$ where $\b{\xi, \xi^*} = \alpha ||\xi^*||^2$.
The result follows.

(iii)
Apply (ii) to $\Phi^*$.

(iv)
Using $\xi = E_0 \xi$ and Lemma \ref{lem:nuE0Es0E0} one finds that $\nu^{-1} \xi = E_0 E^*_0 \xi$.
To finish the proof, evaluate $E_0 E^*_0 \xi$ using (ii), (iii).
\end{proof}

\begin{definition}    \label{def:dualstandard}    \samepage
\ifDRAFT {\rm def:dualstandard}. \fi
By a {\em dual $\Phi$-standard basis} of $V$ we mean the dual of a
$\Phi$-standard basis with respect to $\b{ \; , \; }$.
\end{definition}

Shortly we will describe the dual $\Phi$-standard bases.
We will use the following definition.

\begin{definition}    \label{def:partner}    \samepage
\ifDRAFT {\rm def:partner}. \fi
Note that for nonzero $\xi$, $\zeta \in E_0V$ the following are equivalent:
\begin{align*}
& \text{\rm (i)  $\bbig{\xi, \zeta} = \nu$};
&& \text{\rm (ii)  $\zeta = \nu \xi / ||\xi||^2$};
&& \text{\rm (iii)  $\xi =  \nu \zeta / ||\zeta||^2$}.
\end{align*}
We say that $\xi$, $\zeta$ are {\em partners} whenever they satisfy (i)--(iii).
\end{definition}

\begin{lemma}    \label{lem:dualstandard3}    \samepage
\ifDRAFT {\rm lem:dualstandard3}. \fi
For nonzero $\xi$, $\zeta$ in $E_0 V$
the following are equivalent:
\begin{itemize}
\item[\rm (i)]
the bases $\{E^*_i \xi\}_{i=0}^d$ and $\{ k_i^{-1} E^*_i \zeta\}_{i=0}^d$ are dual
with respect to $\b{ \; , \; }$;
\item[\rm (ii)]
$\xi$, $\zeta$ are partners.
\end{itemize}
\end{lemma}

\begin{proof}
The vector space $E_0 V$ has dimension $1$, so there exists a scalar $\alpha$
such that $\zeta = \alpha \xi$.
By this and Lemma \ref{lem:EsiwEsjw},
\[
 \bbig{ E^*_i \xi,  k_j^{-1} E^*_j \zeta } = \delta_{i,j} \alpha \nu^{-1} ||\xi||^2.
\]
So (i) holds if and only if $\alpha ||\xi||^2 = \nu$.
By this and Definition \ref{def:partner} we obtain the result.
\end{proof}

\begin{lemma}    \label{lem:dualstandard}    \samepage
\ifDRAFT {\rm lem:dualstandard}. \fi
A given basis of $V$ is a dual $\Phi$-standard basis if and only if it has the form
\begin{align}
 \{ k_i^{-1} E^*_i \zeta\}_{i=0}^d,  \qquad \qquad 0 \neq \zeta \in E_0 V.  \label{eq:dualstdbasis}
\end{align}
\end{lemma}

\begin{proof}
Use Lemma \ref{lem:dualstandard3}.
\end{proof}

We mention a result for later use.

\begin{lemma}    \label{lem:EsixiEjxis}     \samepage
\ifDRAFT {\rm lem:EsixiEjxis}. \fi
For $0 \neq \xi \in E_0V$ and $0 \neq \xi^* \in E^*_0 V$,
\begin{align*}
  \bbig{ E^*_i \xi, E_j \xi^* } &= \nu^{-1} p_i (j) k^*_j \bbig{ \xi, \xi^* }   &&  (0 \leq i,j \leq d).
\end{align*}
\end{lemma}

\begin{proof}
Using $E_0 \xi = \xi$, $E^*_0 \xi^* = \xi^*$ and 
Lemma \ref{lem:E0EsiEjEs0}(i),
\[
\bbig{E^*_i \xi, E_j \xi^* }
 = \bbig{ \xi , E_0 E^*_i E_j E^*_0 \xi^* }
 = p_i (j) m_j \bbig{\xi, \xi^*}.
\]
By this and Lemma \ref{lem:kimi}(ii) we obtain the result.
\end{proof}

\section{Four bases of $V$}
\label{sec:4bases}

We continue to discuss a symmetric idempotent system
$\Phi = (\{E_i\}_{i=0}^d; \{E^*_i\}_{i=0}^d)$ on $V$.
Recall the elements $A_i$ from Definition \ref{def:Ai}.
Recall the matrices $K$, $K^*$, $U$, $U^*$ from Definitions \ref{def:matrices}, \ref{def:U},
and the matrices $B_i$, $B^*_i$, $H_i$, $H^*_i$ from Definitions \ref{def:matrices2}, \ref{def:matrices3}.
Recall the bilinear form $\b{\; , \;}$ from above Lemma \ref{lem:bilin1}.

Throughout this section, we fix nonzero vectors $\xi, \zeta \in E_0 V$ and $\xi^*, \zeta^* \in E^*_0 V$,
and consider the following four bases of $V$:
\begin{equation}
\begin{array}{c|c}
\text{basis type} & \text{basis}
\\ \hline
\text{$\Phi$-standard}    & \{ E^*_i \xi \}_{i=0}^d             \rule{0mm}{3ex}
\\
\text{dual $\Phi$-standard} & \{ k_i^{-1} E^*_i \zeta\}_{i=0}^d         \rule{0mm}{3ex}
\\
\text{$\Phi^*$-standard} & \{E_i \xi^*\}_{i=0}^d         \rule{0mm}{3ex}
\\
\text{dual $\Phi^*$-standard} &  \{ (k^*_i)^{-1} E_i \zeta^* \}_{i=0}^d         \rule{0mm}{3ex}
\end{array}     
    \label{eq:4bases}
\end{equation}
In this section we display the matrices that represent $\{A_r\}_{r=0}^d$, $\{A^*_r\}_{r=0}^d$, 
$\{E_r\}_{r=0}^d$, $\{E^*_r\}_{r=0}^d$ with respect to these bases.
We display the inner product matrices between these bases.
We display the transition matrices between these bases.

We introduce some notation.
For $0 \leq i,j \leq d$ define $\Delta_{i,j} \in \Mat_{d+1}(\F)$
that has $(i,j)$-entry $1$ and all other entries $0$.

\begin{proposition}    \label{prop:matrixAr}    \samepage
\ifDRAFT {\rm prop:matrixAr}. \fi
In the table below we give some matrix representations.
For $0 \leq r \leq d$, each entry in the table is the matrix that represents the map
in the given column with respect to the basis in the given row:
\[
\begin{array}{c|ccccc}
\text{\rm basis} & & A_r & A^*_r  & E_r & E^*_r
\\ \hline
 \{ E^*_i \xi \}_{i=0}^d
  & & B_r & H^*_r  & \nu^{-1} U K^* \Delta_{r,r} U^* K & \Delta_{r,r}    \rule{0mm}{3ex}
\\
\{ k_i^{-1} E^*_i \zeta \}
 & & B_r^{\sf t} & H^*_r  & (U^*)^{-1} \Delta_{r,r} U^* & \Delta_{r,r}     \rule{0mm}{3ex}
\\
\{ E_i \xi^* \}_{i=0}^d
  & & H_r & B^*_r   & \Delta_{r,r} & \nu^{-1} U^* K \Delta_{r,r} U K^*        \rule{0mm}{3ex}
\\
\{ (k^*_i)^{-1} E_i \zeta^* \}_{i=0}^d
& & H_r & (B^*_r)^{\sf t}    & \Delta_{r,r} & U^{-1} \Delta_{r,r} U                     \rule{0mm}{3ex}
\end{array}
\]
\end{proposition}

\begin{proof}
We first consider the matrices representing $A_r$.
The matrix representing $A_r$ with respect to $\{E^*_i \xi\}_{i=0}^d$ is obtained using
Lemma \ref{lem:AjEsiE0}(i) and Definition \ref{def:matrices2}.
The matrix representing $A_r$ with respect to $\{k_i^{-1} E^*_i \zeta\}_{i=0}^d$ is obtained
using Lemmas \ref{lem:khphij}(i) and \ref{lem:AjEsiE0}(i).
The matrices representing $A_r$ with respect to $\{E_i \xi^*\}_{i=0}^d$ and $\{(k^*_i)^{-1} E_i \zeta^*\}_{i=0}^d$ 
are obtained using Lemma \ref{lem:AiEj}(i) and Definition \ref{def:matrices3}.
Applying these result to $\Phi^*$ we obtain the matrices representing $A^*_r$.
Next we consider the matrices representing $E_r$.
The matrix representing $E_r$ with respect to $\{E^*_i \xi\}_{i=0}^d$ is obtained using
Lemmas  \ref{lem:E0EsiEj}(iii), \ref{lem:Asi}, \ref{lem:U3}(i),(iii).
Multiply this matrix on the left (resp.\ right) by $K$ (resp.\  $K^{-1}$)
and use Lemma \ref{lem:U2}(iii) to obtain the matrix representing $E_r$
with respect to $\{k_i^{-1} E^*_i \zeta\}_{i=0}^d$.
The matrices representing $E_r$ with respect to $\{E_i \xi^*\}_{i=0}^d$ and $\{ (k^*_i)^{-1} E_i \zeta^* \}_{i=0}^d$
are routinely obtained.
Applying these results to $\Phi^*$ we obtain the matrices representing $E^*_r$.
\end{proof}

\begin{proposition}   \label{prop:innermatrix0}    \samepage
\ifDRAFT {\rm prop:innermatrix0}. \fi
In the table below we give the inner product matrices between the bases in \eqref{eq:4bases}.
Each entry of the table is the inner product matrix from the basis in the given row
to the basis in the given column:
\[
\begin{array}{c|cccc}
 & \{ E^*_i \xi \}_{i=0}^d & \{ k_i^{-1}  E^*_i \zeta \}_{i=0}^d 
   & \{ E_i \xi^* \}_{i=0}^d & \{ (k^*_i)^{-1} E_i \zeta^* \}_{i=0}^d
\\ \hline
\{ E^*_i \xi \}_{i=0}^d 
  & \frac{ ||\xi||^2 } { \nu } K 
  & \frac{ \b{ \xi, \zeta} } { \nu }  I 
  & \frac{ \b{\xi, \xi^*} } { \nu } K U K^* 
  & \frac{ \b{\xi, \zeta^*} } { \nu }   K U  \rule{0mm}{3.5ex}
\\
\{ k_i^{-1} E^*_i \zeta \}_{i=0}^d 
  &  \frac{ \b{ \zeta, \xi} } { \nu }   I 
  & \frac{ || \zeta ||^2 } { \nu  } K^{-1} 
  & \frac{ \b{\zeta, \xi^*} } { \nu } U K^* 
  & \frac{ \b{ \zeta, \zeta^*}  } { \nu } U  \rule{0mm}{3ex}
\\
\{ E_i \xi^* \}_{i=0}^d 
  &   \frac{ \b{ \xi^*, \xi } } { \nu } K^* U^* K
  & \frac{ \b{ \xi^* , \zeta } } { \nu } K^* U^*
  & \frac{ ||\xi^*||^2 } { \nu } K^* 
  & \frac{ \b{ \xi^*, \zeta^* } } { \nu }  I    \rule{0mm}{3.5ex}
\\
\{ (k^*_i)^{-1} E_i \zeta^* \}_{i=0}^d 
  & \frac{ \b{ \zeta^*, \xi } } { \nu } U^* K 
  &  \frac{ \b{ \zeta^*, \zeta}  } { \nu } U^*
  &  \frac{ \b{ \zeta^*, \xi^* } } { \nu }  I 
  & \frac{ ||\zeta^*||^2 } { \nu } (K^*)^{-1}   \rule{0mm}{3.5ex}
\end{array}
\]
\end{proposition}

\begin{proof}
Note that $\zeta$ (resp.\ $\zeta^*$) is a nonzero scalar multiple of $\xi$
(resp.\ $\xi^*$).
Using this and Lemmas \ref{lem:EsiwEsjw}, \ref{lem:EsixiEjxis}
we represent the inner products in terms of $P$, $Q$, $K$, $K^*$.
Now eliminate $P$, $Q$ using Lemma \ref{lem:U3} to get the result.
\end{proof}

In the diagram below we display the inner product matrices between
the four bases in \eqref{eq:4bases}:
\vspace{1.5cm}
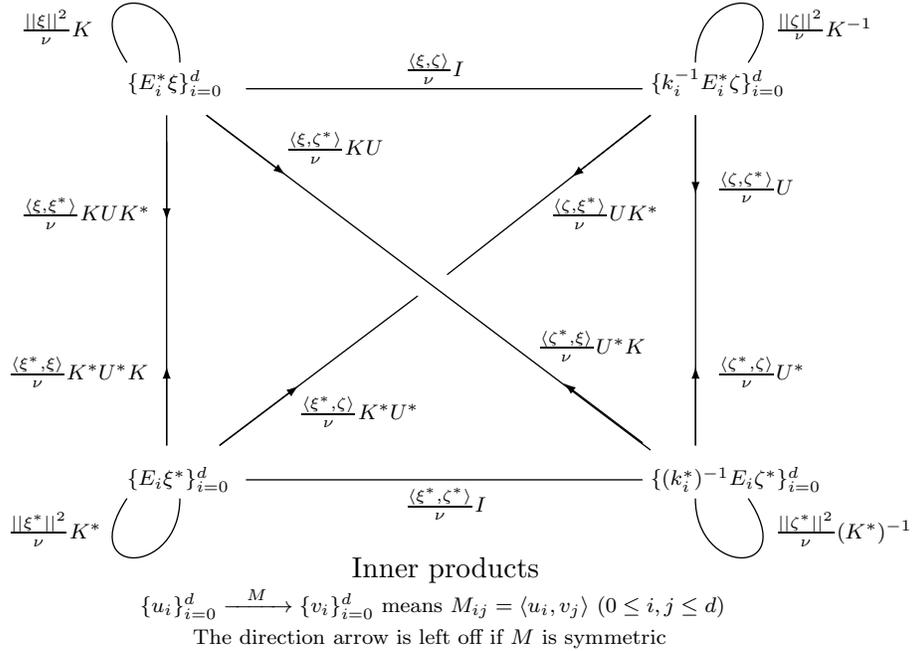
\begin{center}
\scriptsize
\begin{picture}(300,150)
 \linethickness{0.5pt}
 \put(-5,0){$\{E_i \xi^* \}_{i=0}^d$}
 \put(-5,150){$\{E^*_i \xi \}_{i=0}^d$}
 \put(193,150){$\{k_i^{-1} E^*_i \zeta \}_{i=0}^d$}
 \put(193,0){$\{ (k^*_i)^{-1} E_i \zeta^* \}_{i=0}^d$}
 \put(30,15){\line(53,40){75}}
 \put(30,15){\vector(53,40){30}}
 \put(40,2){\line(1,0){150}}
 \put(40,150){\line(1,0){150}}
 \put(25,140){\vector(53,-40){30}}
 \put(10,15){\line(0,1){125}}
 \put(10,15){\vector(0,1){30}}
 \put(10,130){\vector(0,-1){30}}
 \put(210,15){\line(0,1){125}}
 \put(193,140){\line(-51,-40){77}}
 \put(193,140){\vector(-51,-40){30}}
 \put(210,140){\vector(0,-1){30}}
 \put(25,140){\line(53,-40){168}}
 \put(190,16){\vector(-53,40){30}}
 \put(210,15){\vector(0,1){30}}
\cbezier(15,160)(15,190)(-25,190)(-5,160)
\cbezier(210,160)(210,190)(250,190)(230,160)
\cbezier(210,-5)(210,-35)(250,-35)(230,-5)
\cbezier(15,-5)(15,-35)(-25,-35)(-5,-5)
 \put(80,-35){\normalsize \rm Inner products}
 \put(0,-48){$\{u_i\}_{i=0}^d \xrightarrow{\;\;\;M\;\;\;} \{v_i\}_{i=0}^d$
   {\rm means} $M_{i j} = \b{u_i, v_j}$ $(0 \leq i,j \leq d)$}
 \put(20,-60){\rm The direction arrow is left off if $M$ is symmetric}
 \put(100,-10){$ \frac{ \b{ \xi^*, \zeta^* } } { \nu } I$}
 \put(60,25){$\frac{ \b{\xi^*, \zeta} } { \nu } K^* U^* $}
 \put(-50,40){$\frac{ \b{\xi^*, \xi} } { \nu } K^* U^* K$}
 \put(-45,100){$\frac{ \b{ \xi, \xi^* } } { \nu } K U K^*$}
 \put(55,125){$ \frac{ \b{ \xi, \zeta^* }   } { \nu } K U$}
 \put(100,155){$ \frac{ \b{ \xi, \zeta } } { \nu } I $}
 \put(218,40){$\frac{ \b{ \zeta^*, \zeta}  } { \nu } U^*$}
 \put(218,110){$ \frac{ \b{ \zeta, \zeta^*}  } { \nu } U $}
 \put(150,50){$ \frac{ \b{ \zeta^*, \xi }  } { \nu  } U^* K$}
 \put(155,100){$ \frac{ \b{ \zeta, \xi^* } } { \nu  } U K^*$}
 \put(-45,170){$\frac{ ||\xi||^2 } { \nu } K$}
 \put(240,170){$ \frac{ ||\zeta||^2 } { ^\nu } K^{-1}$}
 \put(-50, -20){$\frac{ ||\xi^*||^2 } { \nu } K^*$}
 \put(240,-20){$\frac{ ||\zeta^*||^2 } { \nu }  (K^*)^{-1}$}
\end{picture}
\end{center}

\vspace{1.5cm}

\begin{proposition}         \label{prop:transmatrix0}   \samepage
\ifDRAFT {\rm prop:transmatrix0}. \fi
In the table below we give the transition matrices between the
four bases in \eqref{eq:4bases}.
Each entry of the table is the transition matrix from the
basis in the given row to the basis in the given column:
\[
\begin{array}{c|cccc}
 & \{ E^*_i \xi \}_{i=0}^d & \{ k_i^{-1} E^*_i \zeta \}_{i=0}^d 
   & \{ E_i \xi^* \}_{i=0}^d & \{ (k^*_i)^{-1} E_i \zeta^* \}_{i=0}^d
\\ \hline
\{ E^*_i \xi \}_{i=0}^d 
& I 
& \frac{ \b{\xi, \zeta} } { ||\xi||^2 } K^{-1} 
& \frac{ \b{\xi, \xi^* } } { ||\xi||^2 } U K^* 
& \frac{ \b{\xi, \zeta^*} } { ||\xi||^2 } U  \rule{0mm}{3ex}
\\
\{ k_i^{-1}  E^*_i \zeta \}_{i=0}^d 
&  \frac{ \b{ \zeta, \xi} } { ||\zeta||^2 } K 
&  I 
& \frac{ \b{ \zeta, \xi^* } } { ||\zeta||^2 } K U K^* 
& \frac{ \b{ \zeta, \zeta^* } } { ||\zeta||^2 } K U  \rule{0mm}{3.5ex}
\\
\{ E_i \xi^* \}_{i=0}^d  
& \frac{ \b{ \xi^*, \xi } } { ||\xi^*||^2 } U^* K 
&  \frac{ \b{ \xi^*, \zeta } } { ||\xi^*||^2 } U^*
& I 
& \frac{ \b{ \xi^*, \zeta^*} } { ||\xi^*||^2 } (K^*)^{-1}   \rule{0mm}{3.5ex}
\\
\{ (k^*_i)^{-1} E_i \zeta^* \}_{i=0}^d 
&  \frac{ \b{ \zeta^*, \xi } } { ||\zeta^*||^2 } K^* U^* K
& \frac{ \b{ \zeta^*, \zeta } } { ||\zeta^*||^2 } K^* U^*
& \frac{ \b{ \zeta^*, \xi^* }  } { ||\zeta^*||^2 } K^* 
& I    \rule{0mm}{3.5ex}
\end{array}
\]
\end{proposition}

\begin{proof}
Use Lemma \ref{lem:trans} and Proposition \ref{prop:innermatrix0}.
\end{proof}

In the diagram below we display the transition matrices between
the four bases in \eqref{eq:4bases}:
\vspace{0.5cm}
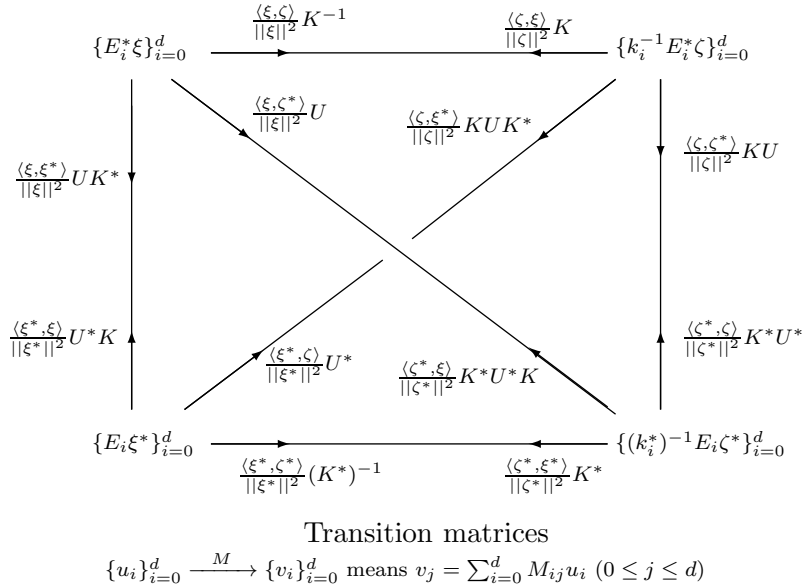
\begin{center}
\scriptsize
\begin{picture}(300,150)
 \linethickness{0.5pt}
 \put(-5,0){$\{E_i \xi^* \}_{i=0}^d$}
 \put(-5,150){$\{E^*_i \xi \}_{i=0}^d$}
 \put(193,150){$\{ k_i^{-1} E^*_i \zeta \}_{i=0}^d$}
 \put(193,0){$\{ (k^*_i)^{-1} E_i \zeta^* \}_{i=0}^d$}
 \put(30,15){\line(53,40){75}}
 \put(30,15){\vector(53,40){30}}
 \put(40,2){\line(1,0){150}}
 \put(40,2){\vector(1,0){30}}
 \put(40,150){\line(1,0){150}}
 \put(40,150){\vector(1,0){30}}
 \put(25,140){\vector(53,-40){30}}
 \put(10,15){\line(0,1){125}}
 \put(10,15){\vector(0,1){30}}
 \put(10,130){\vector(0,-1){30}}
 \put(210,15){\line(0,1){125}}
 \put(193,140){\line(-51,-40){77}}
 \put(193,140){\vector(-51,-40){30}}
 \put(190,150){\vector(-1,0){30}}
 \put(210,140){\vector(0,-1){30}}
 \put(25,140){\line(53,-40){168}}
 \put(190,2){\vector(-1,0){30}}
 \put(190,16){\vector(-53,40){30}}
 \put(210,15){\vector(0,1){30}}
 \put(75,-35){\normalsize \rm Transition matrices}
 \put(0,-48){$\{u_i\}_{i=0}^d \xrightarrow{\;\;\;M\;\;\;} \{v_i\}_{i=0}^d$
   {\rm means} $v_j = \sum_{i=0}^d M_{i j} u_i$ $(0 \leq j \leq d)$}
 \put(50,-12){$\frac{ \b{ \xi^*, \zeta^* } } { ||\xi^*||^2 } (K^*)^{-1}$}
 \put(60,30){$\frac{ \b{ \xi^*, \zeta} } { ||\xi^*||^2 } U^*$}
 \put(-37,40){$\frac{  \b{ \xi^*, \xi } } { ||\xi^*||^2 } U^* K$}
 \put(-35,100){$\frac  { \b{ \xi, \xi^* } } { ||\xi||^2 }  U K^*$}
 \put(55,125){$\frac{ \b{ \xi, \zeta^*}  } { ||\xi||^2 } U$}
 \put(55,160){$\frac{ \b{ \xi, \zeta } } { ||\xi||^2 } K^{-1}$}
 \put(150,157){$ \frac{ \b{ \zeta, \xi} } { ||\zeta||^2 }  K$}
 \put(150,-12){$ \frac{ \b{ \zeta^*, \xi^* }  } { ||\zeta^*||^2 } K^* $}
 \put(218,40){$\frac{ \b{ \zeta^*, \zeta }  }  { ||\zeta^*||^2 } K^* U^* $}
 \put(218,110){$\frac{ \b{ \zeta,  \zeta^* } } { ||\zeta||^2 } K U $}
 \put(110,25){$ \frac{ \b{ \zeta^*, \xi} } { ||\zeta^*||^2 }   K^* U^* K$}
 \put(113,120){$\frac{ \b{ \zeta, \xi^* } } { ||\zeta||^2 }  K U K^*$}
\end{picture}
\end{center}

\vspace{1.5cm}

\section{$P$-polynomial and $Q$-polynomial idempotent systems}
\label{sec:Ppoly}

We continue to discuss a symmetric idempotent system
$\Phi = (\{E_i\}_{i=0}^d; \{E^*_i\}_{i=0}^d)$ on $V$.

\begin{definition}    \label{def:Ppoly}     \samepage
\ifDRAFT {\rm def:Ppoly}. \fi
We say that $\Phi$ is {\em $P$-polynomial}
whenever $p^h_{i j}$ is zero (resp.\ nonzero) if one of $h,i,j$ is greater
than (resp.\ equal to) the sum of the other two $(0 \leq h,i,j \leq d)$.
\end{definition}

For the moment, assume that $d \geq 1$ and $\Phi$ is $P$-polynomial.
Then the first intersection matrix $B_1$ has the form
\[
 B_1 =
  \begin{pmatrix}
    a_0 & b_0 &    & & & \text{\bf 0}                  \\
    c_1 & a_1 & b_1   \\
         & c_2  & \cdot & \cdot  \\
         &      & \cdot & \cdot & \cdot \\
         &       &         & \cdot & \cdot & b_{d-1} \\
     \text{\bf 0}   &        &          &        & c_d & a_d   \\
  \end{pmatrix},
\]
where 
\begin{align*}
c_i &= p^{i}_{1, i-1} \quad (1 \leq i \leq d),
&
a_i &= p^i_{1, i}   \quad (0 \leq i \leq d),
&
b_i &= p^{i}_{1, i+1} \quad (0 \leq i \leq d-1). 
\end{align*}
Moreover $c_i \neq 0$ for $1 \leq i \leq d$ and
$b_i \neq 0$ for $0 \leq i \leq d-1$.
So $B_1$ is irreducible tridiagonal.
Shortly we will show that this feature of $B_1$ characterizes the $P$-polynomial property.

\begin{lemma}    \label{lem:A1Ai}    \samepage
\ifDRAFT {\rm lem:A1Ai}. \fi
Assume that $d \geq 1$ and $\Phi$ is $P$-polynomial.
Then
\begin{align*}
 A_1 A_0 &= a_0 A_0 + c_1 A_1,
\\
 A_1 A_i &= b_{i-1} A_{i-1} + a_i A_i + c_{i+1} A_{i+1}  && (1 \leq i \leq d-1),  
\\
 A_1 A_d &= b_{d-1} A_{d-1} + a_d A_d.
\end{align*}
\end{lemma}

\begin{proof}
By Lemma \ref{lem:defp} and the comments below Definition \ref{def:Ppoly}.
\end{proof}

For elements $A$, $B$ in any algebra,
we say that $B$ is an {\em affine transformation} of $A$ whenever
there exist scalars $\alpha$, $\beta$ such that $\alpha \neq 0$
and $B = \alpha A + \beta I$.

\begin{proposition}    \label{prop:Ppoly}    \samepage
\ifDRAFT {\rm prop:Ppoly}. \fi
Assume that $d \geq 1$.
Then for $A \in \text{\rm End}(V)$ the following are equivalent:
\begin{itemize}
\item[\rm (i)]
$\Phi$ is $P$-polynomial and $A$ is an affine transformation of $A_1$;
\item[\rm (ii)]
for $0 \leq i \leq d$ there exists $f_i \in \F[x]$ such that $\text{\rm deg}(f_i)=i$
and $A_i = f_i (A)$.
\end{itemize}
\end{proposition}

\begin{proof}
(i) $\Rightarrow$ (ii)
By Lemma \ref{lem:A1Ai} and since $A_0 = I$.

(ii) $\Rightarrow$ (i)
The elements $\{A_i\}_{i=0}^d$ are linearly independent by Lemma \ref{lem:AiAsi},
so the elements $\{A^i\}_{i=0}^d$ are linearly independent.
Pick integers $i$, $j$ $(0 \leq i,j \leq d)$ such that $i+j \leq d$.
We show that
\begin{align}
  f_i  f_j &= \sum_{h=0}^d p^h_{i j}  f_h       \label{eq:fifj}
\end{align}
Define a polynomial $g = f_i f_j - \sum_{h=0}^d p^h_{i j} f_h$.
The degree of $g$ is at most $d$,  and $g(A)=0$.
Therefore $g=0$.
We have shown \eqref{eq:fifj}.
In \eqref{eq:fifj} we examine the degrees to find
\[
   i + j = \text{\rm max} \{h \,|\, 0 \leq h \leq d, \; p^h_{i j} \neq 0 \}.
\]
By this and Lemma \ref{lem:khphij}(i),
we find that $\Phi$ is $P$-polynomial.
Since $A_1 = f_1 (A)$ and $\text{\rm deg} (f_1)= 1$,
$A$ is an affine transformation of $A_1$.
\end{proof}

\begin{proposition}     \label{prop:Ppoly2}    \samepage
\ifDRAFT {\rm prop:Ppoly2}. \fi
Assume that $d \geq 1$ and $\Phi$ is $P$-polynomial.
Then the following hold:
\begin{itemize}
\item[\rm (i)]
$\{A_1^i\}_{i=0}^d$ form a basis for the vector space $\cal M$, 
where $\cal M$ is from Definition \ref{def:D};
\item[\rm (ii)]
$\{p_1 (j)\}_{j=0}^d$ are mutually distinct;
\item[\rm (iii)]
$\{E_i V\}_{i=0}^d$ are the eigenspaces of $A_1$;
\item[\rm (iv)]
$A_1$ is multiplicity-free;
\item[\rm (v)]
$\{E_i\}_{i=0}^d$ are the primitive idempotents of $A_1$.
\end{itemize}
\end{proposition}

\begin{proof}
(i)
By Lemma \ref{lem:AiAsi} and Proposition \ref{prop:Ppoly}(ii).

(ii)
By Lemma \ref{lem:AiEj}, $p_1 (j)$ is the eigenvalue of $A_1$ 
corresponding to $E_j V$ for $0 \leq j \leq d$.
So the characteristic polynomial of $A_1$ is $\prod_{j=0}^d (x - p_1 (j))$.
By (i) the minimal polynomial of $A_1$ has degree $d+1$.
By these comments, the minimal polynomial of $A_1$ is $\prod_{j=0}^d (x - p_1 (j))$.
The result follows.

(iii)
By Lemma \ref{lem:AiEj}(i) and (ii) above.

(iv)
By (iii) above and since $E_i V$ has dimension one for $0 \leq i \leq d$.

(v)
By (iii), (iv) above. 
\end{proof}

\begin{proposition}    \label{prop:PpolyB1}    \samepage
\ifDRAFT {\rm prop:PpolyB1}. \fi
For $d \geq 1$ the following are equivalent:
\begin{itemize}
\item[\rm (i)]
$\Phi$ is $P$-polynomial;
\item[\rm (ii)]
the first intersection matrix $B_1$ is irreducible tridiagonal.
\end{itemize}
\end{proposition}

\begin{proof}
(i) $\Rightarrow$ (ii)
We saw this above Lemma \ref{lem:A1Ai}.

(ii) $\Rightarrow$ (i)
Since $B_1$ is irreducible tridiagonal, we have the equations in Lemma \ref{lem:A1Ai}.
So for $0 \leq i \leq d$ there exists $f_i \in \F[x]$ such that $\text{\rm deg}(f_i) = i$ 
and $A_i = f_i (A_1)$.
By Proposition \ref{prop:Ppoly} (with $A = A_1$) we see that $\Phi$ is $P$-polynomial.
\end{proof}

\begin{definition}    \label{def:Qpoly}    \samepage
\ifDRAFT {\rm def:Qpoly}. \fi
We say that $\Phi$ is {\em $Q$-polynomial}
whenever $q^h_{i j}$ is zero (resp.\ nonzero) if one of $h,i,j$ is greater
than (resp.\ equal to) the sum of the other two $(0 \leq h,i,j \leq d)$.
\end{definition}

\begin{lemma}    \label{lem:Qpoly}   \samepage
\ifDRAFT {\rm lem:Qpoly}. \fi
$\Phi$ is $Q$-polynomial if and only if $\Phi^*$ is $P$-polynomial.
\end{lemma}

\begin{proof}
Immediate from Definitions  \ref{def:qhij},
\ref{def:Ppoly}, \ref{def:Qpoly}.
\end{proof}

\section{Leonard pairs and Leonard systems}
\label{sec:LP}

In this section we recall the notion of a Leonard pair and a Leonard system.

\begin{definition}  \cite[Definition 1.1]{T:Leonard}
\label{def:LP}    \samepage
\ifDRAFT {\rm def:LP}. \fi
By a {\em Leonard pair} on $V$ we mean an ordered pair
$A,A^*$ of elements in $\text{\rm End}(V)$
that satisfy the following {\rm (i), (ii)}.
\begin{itemize}
\item[\rm (i)]
There exists a basis of $V$ with respect to which the matrix representing $A$
is irreducible tridiagonal and the matrix representing $A^*$ is diagonal.
\item[\rm (ii)]
There exists a basis of $V$ with respect to which the matrix representing $A^*$
is irreducible tridiagonal and the matrix representing $A$ is diagonal.
\end{itemize}
\end{definition}

Let $A,A^*$ denote a Leonard pair on $V$.
By\cite[Lemma 1.3]{T:Leonard} each of $A$, $A^*$ is multiplicity-free.
Let $\{E_i\}_{i=0}^d$ denote an ordering of the primitive idempotents of $A$.
For $0 \leq i \leq d$ pick a nonzero $v_i \in E_i V$.
Then $\{v_i\}_{i=0}^d$ form a basis of $V$.
We say that the ordering $\{E_i\}_{i=0}^d$ is {\em standard} whenever
$\{v_i\}_{i=0}^d$ satisfies Definition \ref{def:LP}(ii).
In this case, the ordering $\{E_{d-i} \}_{i=0}^d$ is standard and no further
ordering is standard.
A standard ordering of the primitive idempotents of $A^*$ is similarly defined.

\begin{definition}   {\rm \cite[Definition 1.4]{T:Leonard} }
 \label{def:LS}   \samepage
\ifDRAFT {\rm def:LS}. \fi
By a {\em Leonard system} on $V$  we mean a sequence
\begin{equation}
   (A; \{E_i\}_{i=0}^d; A^*; \{E^*_i\}_{i=0}^d)     \label{eq:Phi}
\end{equation}
of elements in $\text{\rm End}(V)$
that satisfy the following (i)--(iii):
\begin{itemize}
\item[\rm (i)]
$A,A^*$ is a Leonard pair on $V$;
\item[\rm (ii)]
$\{E_i\}_{i=0}^d$ is a standard ordering of the primitive idempotents of $A$;
\item[\rm (iii)]
$\{E^*_i\}_{i=0}^d$ is a standard ordering of the primitive idempotents of $A^*$.
\end{itemize}
\end{definition}

For the rest of this section
let $(A; \{E_i\}_{i=0}^d; A^*; \{E^*_i\}_{i=0}^d)$ denote a Leonard system on $V$.
Note that $(A^*; \{E^*_i\}_{i=0}^d; A; \{E_i\}_{i=0}^d)$ is a Leonard system on $V$.
 
\begin{lemma}   {\rm   \cite[Lemma 9.2]{T:qRacah} }
 \label{lem:E0EsiE02}    \samepage
\ifDRAFT {\rm lem:E0EsiE02}. \fi
The following hold:
\begin{itemize}
\item[\rm (i)]
$E_0 E^*_i E_0  \neq 0 \quad (0 \leq i \leq d)$;
\item[\rm (ii)]
$E^*_0 E_i E^*_0 \neq 0 \quad (0 \leq i \leq d)$.
\end{itemize}
\end{lemma}

\begin{lemma}    {\rm \cite[Theorem 6.1 and Lemma 6.3]{T:qRacah} }
\label{lem:anti}    \samepage
\ifDRAFT {\rm lem::anti}. \fi
There exists a unique antiautomorphism $\dagger$ of $\text{\rm End}(V)$
that fixes each of $A$, $A^*$.
Moreover $\dagger$ fixes each of $E_i$, $E^*_i$  for $0 \leq i \leq d$.
\end{lemma}

\begin{lemma}    {\rm \cite[Theorem 13.4]{T:qRacah} }
\label{lem:vi}    \samepage
\ifDRAFT {\rm lem:vi}. \fi
There exist polynomials $\{f_i\}_{i=0}^d$ in $\F[x]$ such that
$\text{\rm deg}(f_i) = i$ and $f_i (A) E^*_0 E_0 = E^*_i E_0$
for $0 \leq i \leq d$.
\end{lemma}

\begin{lemma}   {\rm \cite[Theorem 4.2]{NT:split} }
\label{lem:NTsplit}   \samepage
\ifDRAFT {\rm lem:NTsplit}. \fi
For elements $B$, $B^*$ in $\text{\rm End}(V)$ the following are equivalent:
\begin{itemize}
\item[\rm (i)]
$(B; \{E_i\}_{i=0}^d; B^*; \{E^*_i\}_{i=0}^d)$ is a Leonard system;
\item[\rm (ii)]
$B$ (resp.\ $B^*$) is an affine transformation of $A$ (resp.\ $A^*$).
\end{itemize}
\end{lemma}

\section{Idempotent systems and Leonard systems}
\label{sec:IPSLS}

In this section we show that a Leonard system is essentially the same thing
as a symmetric idempotent system that is $P$-polynomial and $Q$-polynomial.

\begin{theorem}     \label{thm:LP}     \samepage
\ifDRAFT {\rm thm:LP}. \fi
Let  $\Phi = (\{E_i\}_{i=0}^d; \{E^*_i\}_{i=0}^d)$ denote a sequence
of elements in $\text{\rm End}(V)$.
Then the following are equivalent:
\begin{itemize}
\item[\rm (i)]
$\Phi$ is a symmetric idempotent system that is $P$-polynomial and $Q$-polynomial;
\item[\rm (ii)]
there exist $A$, $A^*$ in $\text{\rm End}(V)$ such that
$(A; \{E_i\}_{i=0}^d; A^*; \{E^*_i\}_{i=0}^d)$ is a Leonard system.
\end{itemize}
\end{theorem}

\begin{proof}
We assume $d \geq 1$; otherwise the assertion is obvious.

(i) $\Rightarrow$ (ii)
We show that  $(A_1; \{E_i\}_{i=0}^d; A^*_1; \{E^*_i\}_{i=0}^d)$ is
a Leonard system on $V$, where $A_1$, $A^*_1$ are from
Definition \ref{def:Ai}.
By Proposition \ref{prop:matrixAr}, with respect to a $\Phi$-standard basis of $V$
the matrix representing $A_1$ is $B_1$ 
and the matrix representing $A^*_1$ is $H^*_1$.
By Definition \ref{def:matrices2} the matrix $H^*_1$ is diagonal,
and by Proposition \ref{prop:PpolyB1} the matrix $B_1$ is irreducible tridiagonal.
Thus with respect to a $\Phi$-standard basis the matrix representing $A_1$ is irreducible tridiagonal
and the matrix representing $A^*_1$ is diagonal.
Applying this to $\Phi^*$, with respect to a $\Phi^*$-standard basis the matrix representing $A^*_1$
is irreducible tridiagonal and the matrix representing $A_1$ is diagonal.
By these comments $A_1, A^*_1$ is a Leonard pair on $V$.
By Proposition \ref{prop:Ppoly2}(v) and the construction,
$\{E_i\}_{i=0}^d$ (resp.\ $\{ E^*_i\}_{i=0}^d$)  is a standard ordering of the primitive idempotents
of $A_1$ (resp.\ $A^*_1$).
We have shown that $(A_1; \{E_i\}_{i=0}^d; A^*_1; \{E^*_i\}_{i=0}^d)$ is a Leonard system on $V$.

(ii) $\Rightarrow$ (i)
By Lemmas \ref{lem:E0EsiE02} and \ref{lem:anti}, $\Phi$ is a symmetric idempotent system on $V$.
By Lemma \ref{lem:vi} there exist polynomials $\{f_i\}_{i=0}^d$ in $\F[x]$ such that
$\text{\rm deg}(f_i) = i$ and 
$f_i (A) E^*_0 E_0 = E^*_i E_0$ for $0 \leq i \leq d$.
By Lemmas \ref{lem:rho}, \ref{lem:rhorhospre}, \ref{lem:AiEs0E0}(i),
$A_i$ is the unique element in $\cal M$ such that $A_i E^*_0 E_0 = E^*_i E_0$
$(0 \leq i \leq d)$.
By these comments $f_i (A) = A_i$ for $0 \leq i \leq d$.
By this and Proposition \ref{prop:Ppoly}, $\Phi$ is $P$-polynomial.
Apply this to the Leonard system  $(A^*; \{E^*_i\}_{i=0}^d; A; \{E_i\}_{i=0}^d)$
to find that $\Phi$ is $Q$-polynomial.
\end{proof}

\begin{lemma}   \label{lem:LP2}    \samepage
\ifDRAFT {\rm lem:LP2}. \fi
Assume that $d \geq 1$ and the equivalent conditions {\rm (i), (ii)} hold in Theorem \ref{thm:LP}.
Then for $A$, $A^*$ in $\text{\rm End}(V)$ the following are equivalent:
\begin{itemize}
\item[\rm (i)]
$(A; \{E_i\}_{i=0}^d; A^*; \{E^*_i\}_{i=0}^d)$ is a Leonard system on $V$;
\item[\rm (ii)]
$A$ (resp.\ $A^*$) is an affine transformation of $A_1$ (resp.\ $A^*_1$),
where $A_1$, $A^*_1$ are from Definition \ref{def:Ai}.
\end{itemize}
\end{lemma}

\begin{proof}
(i) $\Rightarrow$ (ii)
In the proof of Theorem \ref{thm:LP} we have shown that
$(A_1; \{E_i\}_{i=0}^d; A^*_1; \{E^*_i\}_{i=0}^d)$ is a Leonard system on $V$.
By this and Lemma \ref{lem:NTsplit}, $A$ (resp.\ $A^*$) is an
affine transformation of $A_1$ (resp.\ $A^*_1$).

(ii) $\Rightarrow$ (i)
By Lemma \ref{lem:NTsplit}.
\end{proof}

\bigskip\bigskip\noindent
Kazumasa Nomura\\
Tokyo Medical and Dental University\\
Kohnodai, Ichikawa, 272-0827 Japan\\
email: knomura@pop11.odn.ne.jp

\bigskip\noindent
Paul Terwilliger\\
Department of Mathematics\\
University of Wisconsin\\
480 Lincoln Drive\\ 
Madison, Wisconsin, 53706 USA\\
email: terwilli@math.wisc.edu

\bigskip\noindent
{\bf Keywords.}
Association scheme, Leonard pair, Bose-Mesner algebra
\\
\noindent
{\bf 2010 Mathematics Subject Classification.} 17B37, 15A21

\end{document}